\documentclass[11pt,article,reqno]{amsart}
\usepackage[top=25mm,bottom=25mm,left=25mm,right=25mm]{geometry}
\usepackage{amsmath,amssymb,amsfonts,amsthm}
\usepackage{pgf,tikz}

\usetikzlibrary{arrows}
\usetikzlibrary{calc}
\usepackage[utf8]{inputenc}

\makeatletter
\def\section{\@startsection{section}{1}%
\z@{1\linespacing\@plus\linespacing}{1\linespacing}%
{\bf\centering}}
\def\subsection{\@startsection{subsection}{0}%
\z@{\linespacing\@plus\linespacing}{\linespacing}%
{\bf}}
\def\subsubsection{\@startsection{subsubsection}{0}%
\z@{\linespacing\@plus\linespacing}{\linespacing}%
{\bf}}
\makeatother

\makeatletter
\@addtoreset{equation}{section}
\makeatother

\newtheorem{assumption}{Assumption}[section]
\newtheorem{definition}{Definition}[section]
\newtheorem{theorem}{Theorem}[section]
\newtheorem{proposition}{Proposition}[section]

\newtheorem{remark}{Remark}[section]
\newtheorem{lemma}{Lemma}[section]

\def\RR{\mathbb{R}}\def\R{\mathbb{R}}\def\rr{\mathbb{R}}

\def\NN{\mathbb{N}}\def\N{\mathbb{N}}\def\nn{\mathbb{N}}

\def\PP{\mathbb{P}}\def\pp{\mathbb{P}}
\def\E{\mathbb{E}}\def\ex{\mathbb{E}}
\def\al{\alpha}
\def\be{\beta}
\def\ga{\gamma}
\def\de{\delta}

\def\ep{\varepsilon}\def\e{\varepsilon}
\def\ph{\varphi}
\def\la{\lambda}
\def\sk{\smallskip}
\def\mk{\medskip}

\def\wt{\widetilde}

\def\cH{\mathcal{H}}
\def\rd{\mathrm{d}}
\def\cK{\mathcal{K}}
\def\cF{\mathcal{F}}

\def\deg{{\rm deg}}
\def\sgn{{\rm sgn}}
\def\Ran{{\rm Ran}}
\def\dimh{\dim_{\tiny  \sf H}}
\def\indiq{\mathbf{1}}
\def \( {\left(}
\def \) {\right)}
\newcommand{\abs}[1]{\left\vert#1\right\vert}
\newcommand{\bra}[1]{\left\lbrace#1\right\rbrace}

\providecommand{\pro}[1]{(#1_t)_{t \geq 0}}
\providecommand{\proo}[1]{(#1_t)_{t \in \R}}
\providecommand{\semi}[1]{\{#1_t: t \geq 0\}}

\DeclareMathOperator{\Dom}{Dom}
\DeclareMathOperator{\Spec}{Spec}

\begin{document}
\title[Multifractality of ground state-transformed jump processes]
{Multifractal properties of sample paths of ground state-transformed jump processes}
\author[J. L\H orinczi and X. Yang]{J\'ozsef L\H orinczi and Xiaochuan Yang}
\address{J\'ozsef L\H orinczi,
Department of Mathematical Sciences, Loughborough University \\
Loughborough LE11 3TU, United Kingdom}
\email{J.Lorinczi@lboro.ac.uk}

\address{Xiaochuan Yang, Department of Statistics and Probability \\
Michigan State University
\\ 619 Red Cedar Road, East Lansing, MI, USA}
\email{yangxi43@stt.msu.edu, xiaochuan.j.yang@gmail.com}

\thanks{
 \\ \medskip
\noindent
2010 {\it MS Classification}: Primary 60J75, 60G17, 28A78; Secondary 47D08, 47G20 \\
\noindent
}

\begin{abstract}
We consider a class of L\'evy-type processes with unbounded coefficients, arising as Doob $h$-transforms of
Feynman-Kac type representations of non-local Schr\"odinger operators, where the function $h$ is chosen to
be the ground state of such an operator. First we show existence of a c\`adl\`ag version of the so-obtained
ground state-transformed processes. Next we prove that they satisfy a related stochastic differential equation
with jumps. Making use of this SDE, we then derive and prove the multifractal spectrum of local H\"older
exponents of sample paths of ground state-transformed processes.
\end{abstract}

\maketitle

\baselineskip 0.5 cm

\bigskip \medskip

\section{Introduction}
The purpose of this paper is to investigate the local H\"older continuity properties of sample paths of a class
of L\'evy-type processes. These processes are obtained through a Doob $h$-transform of random processes occurring
in the Feynman-Kac representation of non-local Schr\"odinger operators. Such processes, and the related operators
and non-local equations, are currently much used in a variety of applications, for instance, in models of
mathematical physics (anomalous diffusion in porous media, quantum optics etc), however, we will not be concerned
with applications in this paper and for a discussion we refer to \cite{KL16b}. Apart from a direct theoretical
relevance and applications, sample path regularity properties are also of practical interest in modelling and
numerical simulations.

Below we will consider $\R^d$-valued L\'evy processes $\pro X$, generated by operators of the form
\begin{equation}
\label{defL}
L f(x) = \sum_{i,j=1}^d a_{ij} \frac{\partial^2 f}{\partial x_j \partial x_i} (x)
+ \int_{\R^d} \left(f(x+z)-f(x)-1_{\{|z|<1\}}(z) \, z \cdot \nabla f(x)\right)\nu(z) dz, \quad x \in \R^d,
\end{equation}
with $f \in C_c^{\infty}(\R^d)$, and where the matrix $A=(a_{ij})_{i,j=1,...,d}$ describes the diffusion part, and
the L\'evy measure $\nu(\rd z) = \nu(z)\rd z$ describes the jump part. (For details see Section 2 below.) A landmark
example is the fractional Laplacian $L = -(-\Delta)^{\alpha/2}$, $0 < \alpha < 2$, giving rise to an isotropic
$\alpha$-stable process, which is a specific case of the class $L = -\Psi(-\Delta)$, where $\Psi$ is a Bernstein
function. Further cases of much interest include jump-diffusion processes obtained as the sum of a mutually independent
Brownian motion and an isotropic $\alpha$-stable process generated by $L = a\Delta -  b(-\Delta)^{\alpha/2}$, $a, b > 0$,
isotropic relativistic stable processes generated by $L = -(-\Delta + m^{2/\alpha})^{\alpha/2} + m$, $m > 0$, isotropic
geometric $\alpha$-stable processes generated by $L =- \log(1 + (-\Delta)^{\alpha/2})$, and many others.

Next we consider a suitable class of Borel functions $V: \R^d \to \R$ called potentials, and define the non-local
Schr\"odinger operator $H = -L +V$ and the related semigroup $\{e^{-tH}: t \geq 0\}$. Assuming that $V$ is a
Kato-class potential, we then have the Feynman-Kac type representation \cite{HIL12,LHB11}
\begin{equation}
\label{FKf}
\left(e^{-tH}f\right)(x) = \ex^x[e^{-\int_0^t V(X_s)ds}f(X_t)] =: T_tf(x), \quad f \in L^2(\R^d), \; x \in \R^d,
\; t \geq 0,
\end{equation}
where the expectation is taken with respect to the probability measure of the L\'evy process $\pro X$. The so obtained
Feynman-Kac semigroup $\semi T$ has many convenient properties, allowing a far reaching study of, for instance,
spectral properties of $H = -L+V$ or solutions of non-local equations of the type $\partial_t u = Hu$. However, it is
not conservative in the sense that $T_t \indiq_{\R^d} \neq \indiq_{\R^d}$, $t > 0$, therefore the L\'evy process $\pro X$
perturbed by the function $V$ is in general no longer a random process. Nevertheless, by a suitable Doob $h$-transform
one can change the measure under which it becomes a Markov process.

Suppose that $H$ has a non-empty discrete component in its spectrum, and let $\varphi_0$ be its unique eigenfunction
(called \emph{ground state}) corresponding to the lowest-lying eigenvalue, i.e., $H\varphi_0 = \lambda_0\varphi_0$ with
$\varphi_0 \in \Dom H$ and $\lambda_0 = \inf \Spec H$. Then the map $f \mapsto \varphi_0f$ defines a unitary transform
from $L^2(\R^d,\varphi_0^2dx)$ to $L^2(\R^d,dx)$. It can be shown, see Section 2 below for further details, that the
image $\widetilde H$ of $H-\lambda_0$ under this unitary map gives the negative of the infinitesimal generator $\wt L$
of a Markov process, and for suitable test functions we have
\begin{multline}
(\widetilde L f)(x) = \frac{1}{2} \sigma\nabla \cdot \sigma\nabla f(x) + \sigma\nabla\ln\ph_0(x)\cdot
\sigma\nabla f(x) + \int_{0<|z|\le 1} \frac{\ph_0(x+z)-\ph_0(x)}{\ph_0(x)} z\cdot \nabla f(x)\nu(z)\rd z \\
+ \int_{\R^d\setminus\{0\}}\big( f(x+z)-f(x)-z\cdot\nabla f(x)\indiq_{\bra{{|z|\le 1}}}
\big)\frac{\ph_0(x+z)}{\ph_0(x)} \nu(z)\rd z,
\label{eq: generator}
\end{multline}
where $\nu$ and $A = \sigma\sigma^T$ are as in \eqref{defL} above, and where we use the notation $\sigma\nabla \cdot
\sigma\nabla f(x) =\sum_{i,j=1}^d (\sigma\sigma^T)_{ij}\partial_{x_i} \partial_{x_j} f(x)$. We call the resulting process
a \emph{ground state-transformed process} (also called \emph{$P(\phi)_1$-process} following the terminology of B. Simon
\cite{S04}).

The ground state-transformed process is a L\'evy-type process resulting from the effect of $V$ giving rise to position-dependent
drift and jump components, having almost surely c\`adl\`ag paths. However, in contrast with many cases of L\'evy-type processes
studied in the literature, the coefficients of $\widetilde L$ are generally unbounded. It is known that pseudo-differential
operators $G$ defined by
$$
(Gf)(x) = -\int_{\R^d} e^{ix\cdot y} g(x,y)\widehat f(y)\rd y, \quad f \in C_{\rm c}^\infty(\R^d),
$$
where the hat means Fourier transform, give rise to L\'evy-type processes under suitable conditions on the symbol $g(x,y)$.
Whenever $C_{\rm c}^\infty(\R^d) \subset \Dom G$ and $G$ generates a Feller process, the Courr\`ege representation
$$
g(x,y) = g(x,0) - i b(x)\cdot y + \frac{1}{2} y \cdot A(x)y + \int_{\R^d\setminus\{0\}} \left(1 - e^{iz\cdot y} + iz \cdot y
 \indiq_{{\bra{|z|\leq 1}}} \right) \nu(x,\rd z)
$$
holds, where the coefficients $b(x)$, $A(x)$ and $\nu(x,\cdot)$ play the same role of drift vector, diffusion matrix, and
jump measure as for L\'evy processes, with the essential difference that they are now position dependent \cite{C64,J,BSW}.
Furthermore, whenever the condition
\begin{equation}
\sup_{x\in\R^d} |g(x,y)| \leq C (1 + |y|^2), \quad y \in \R^d,
\label{bddq}
\end{equation}
holds, with a constant $C > 0$, the symbol can be used to analyze various properties of the process generated by $G$
\cite{SS10,BSW}. It is also known, however, that \eqref{bddq} implies that all of the coefficients $b(x)$, $A(x)$,
$\nu(x, \cdot)$ are bounded \cite{S98}. Recently, there has been an increasing interest in working also with unbounded
coefficients, see \cite{H98,K10,B11,S16,K16} and \cite[Sect. 3.6]{BSW}. The results below on ground state-transformed processes
complement these efforts since our approach is not through an analysis of the symbol, and apart from a direct interest in this
context, our class of processes has an immediate relevance in the study of spectral properties of related self-adjoint operators
and model Hamiltonians as a bonus \cite{HIL13,LHB11}.

Our concern in the present paper is to study sample path regularity properties of ground state-transformed processes
obtained for a large class of operators $H$. The typical long-time behaviour of such processes has been established in
\cite{KL17}, which is useful also in characterizing the support of the related Gibbs path measures defined by the right hand
side of the Feynman-Kac formula (for perturbations of symmetric $\alpha$-stable processes see also \cite{KL12}). While the
asymptotic behaviour on the long run is driven by the large jumps, regularity at short range depends on the ultraviolet properties
of $H$ involving the small jumps. It is reasonable to expect that at least under sufficiently ``nice" potentials $V$ the regularity
of paths of a ground state-transformed process inherits the regularity of the underlying L\'evy process and it does not deteriorate.
However, since the drift generated by the perturbation may become rough, the challenge is to establish conditions on $V$ under which
path regularity is at least preserved. Results in \cite{BFJS10, Y15, Y16}, where $L = (-\Delta)^{s(x)}$ and $V \equiv 0$, i.e.,
stable-like processes generated by fractional Laplacians of variable order are considered, indicate that local behaviour may become
very complex, and instead of an almost sure rule it can be even dependent on the individual path.

To describe local path regularity, we study the multifractal spectrum of local H\"older exponents of paths.
Recall that given a locally bounded function $f: \rr\to \rr^d$, it is said to belong to the pointwise H\"older space $C^h(x_0)$
for $h>0$ and $x_0\in\rr$ whenever there exist constants $c, \de > 0$, and a polynomial $P$ of $\deg P < \lfloor h \rfloor$
such that for $x\in B(x_0,\de)$,
\begin{align*}
|f(x) - P(x-x_0)|\le c|x-x_0|^h.
\end{align*}
The H\"older exponent of $f$ at point $x_0$ is then defined by
\begin{align*}
H_f(x_0) = \sup\{h>0: f\in C^h(x_0) \}.
\end{align*}
Consider the set
$$
E_f(h)=\{x\in\rr : H_f(x)= h\}.
$$
The \emph{multifractal spectrum} of $f$ is the map
$$
D_f:  h\mapsto \dimh E_f(h),
$$
where $\dimh$ denotes Hausdorff dimension, with the convention that $\dimh \emptyset=-\infty$.

The multifractal spectrum of random processes has been studied by various authors. The above objects are now defined
pathwise. For Brownian motion, the H\"older exponent equals $1\over 2$ everywhere \cite{OT74}, giving
\begin{eqnarray*}
\quad
D_B(h) = \left\{
 \begin{array}{ll}
 1 & \mbox{if $h=\frac{1}{2}$} \\
 -\infty & \mbox{otherwise}
\end{array} \right.
\end{eqnarray*}
almost surely, in which case the multifractal reduces to a mono-fractal behaviour. For a general L\'evy process $\pro X$,
\begin{align}
X_t = bt +  {\sigma B_t} + \int_0^t \int_{|z|\le 1} z \wt N(\rd s, \rd z) + \int_0^t \int_{|z|> 1} z  N(\rd s, \rd z),
\label{levy}
\end{align}
where $b \in \rr^d$ is the drift term, {$\sigma\sigma^T$} is the $d\times d$ diffusion matrix, $N$ is a Poisson measure, and $\wt N$ is the
compensated Poisson measure in $\rr^d$ with intensity given by the L\'evy measure $\nu(\rd z)$, the behaviour relates with
the upper Blumenthal-Getoor index \cite{BG61} given by
\begin{align}\label{eq: BG index}
\be_\nu = \inf\Big\{ \ga \geq 0: \int_{|z|\le 1} |z|^\ga \nu(\rd z)<\infty  \Big\},
\end{align}
describing the growth rate of the L\'evy measure around zero. The integrability condition of L\'evy measures implies that
$\be_\nu \in [0,2]$. Jaffard \cite{J99} has proved the following for a L\'evy process with $\be_\nu\in(0,2)$. If
$\sigma\neq 0$, then
\begin{align}
\label{eq: spec with B}
D^1_X(h) =
\begin{cases} \be_\nu  h & \mbox{ if } h<1/2 \\
1 & \mbox{ if } h = 1/2 \\
-\infty & \mbox{ otherwise}
\end{cases}
\end{align}
almost surely, and if $\sigma = 0$, then
\begin{align}
\label{eq: spec pure jump}
D^2_X(h) =
\begin{cases}
\be_\nu  h & \mbox{ if } h\le 1/\be_\nu \\
-\infty & \mbox{ otherwise}
\end{cases}
\end{align}
almost surely. Balan\c ca \cite{B14} has shown that the same result holds also for $\be_\nu=2$. (An example of a
one-dimensional pure jump L\'evy process with $\be_\nu=2$ is one with intensity $\nu(z) =  1/(z^3|\log z|^a)$, $a>1$.)
Extensions to L\'evy fields and time-changed L\'evy processes can be found in \cite{DJ, BS07}.

We also note that there are many further fractal properties of jump processes addressed in the literature. We refer to
\cite{KX05,MX05,KX08,KX12} and the references therein, and for a review see \cite{X04}.

Our main results are as follows. First, in Theorem \ref{th:exphi1} we prove the existence and basic properties of ground
state-transformed processes in the generality considered in this paper. Next in Theorem \ref{uniqueness}, we derive a
stochastic differential equation with jumps related to $\widetilde L$, and show that the ground state-transformed
processes we consider are a weak solution. Using the SDE representation, in Theorem \ref{spectrum} we then obtain the
multifractal spectrum of local H\"older exponents of our class of processes. We find that whenever the process contains
a Brownian component, it has a sweeping effect, and the behaviour is described by \eqref{eq: spec with B}. For cases of
pure jump processes, there is a split in the behaviour according to Blumenthal-Getoor indices lower or higher than 1.
For values $\be_\nu\in[1,2]$ the behaviour is described by \eqref{eq: spec pure jump}, while for $\be_\nu\in(0,1)$ this
happens under an increased regularity of the ground state. In Section 3.3 we also discuss the necessity of this extra
regularity.

\section{Ground state-transformed jump processes}

\subsection{L\'evy processes and perturbations by potentials}
Let $\pro X$ be a rotationally symmetric L\'evy process with values in $\R^d$, $d \geq 1$, i.e., as given by (\ref{levy})
in which we set $b=0$. The probability measure of the process starting at $x \in \R^d$ will be denoted by $\PP^x$, and
expectation with respect to this measure by $\ex^x$. The process $\pro X$ is determined by its characteristic function
$$
\ex^0 \left[e^{i y \cdot X_t}\right] = e^{-t \psi(y)}, \quad y \in \R^d, \ t >0,
$$
with the characteristic exponent given by the L\'evy-Khintchin formula
\begin{align} \label{eq:Lchexp}
\psi(y) = \frac{1}{2} A y \cdot y + \int_{\R^d} (1-\cos(y \cdot z)) \nu(\rd z).
\end{align}
Here $A=(a_{ij})_{1\leq i,j \leq d}=\sigma\sigma^T$ is a symmetric non-negative definite matrix, and $\nu$ is a symmetric
L\'evy measure on $\R^d \backslash \left\{0\right\}$, i.e., $\int_{\R^d} (1 \wedge |z|^2) \nu(\rd z) < \infty$ and $\nu(E)=
\nu(-E)$, for every Borel set $E \subset \R^d \backslash \left\{0\right\}$, thus the L\'evy triplet of the process
is $(0,\frac{1}{2}A,\nu)$. We will assume throughout that the L\'evy measure in \eqref{eq:Lchexp} has infinite mass
and it is absolutely continuous with respect to Lebesgue measure, i.e., $\nu(\R^d \backslash \left\{0\right\})=\infty$
and $\nu(\rd x)=\nu(x)\rd x$, with density $\nu(x) > 0$.

The generator $L$ of the process $\pro X$ is determined by its symbol $\psi$ through
\begin{align}
\label{def:gene}
\widehat{L f}(y) = - \psi(y) \widehat{f}(y), \quad y \in \R^d, \; f \in \Dom(L),
\end{align}
with domain $\Dom(L)=\big\{f \in L^2(\R^d): \psi \widehat f \in L^2(\R^d) \big\}$. It is a negative, non-local,
self-adjoint operator with core $C_{\rm c}^\infty(\R^d)$, and it has the expression \eqref{defL} for $f \in
C_{\rm c}^\infty(\R^d)$.

Next consider the set of functions
\begin{align}
\label{eq:Katoclass}
\cK^X = \Big\{f: \R\to \R^d: \, \mbox{$f$ is Borel measurable and} \,
\lim_{t \downarrow 0} \sup_{x \in \R^d} \ex^x \Big[\int_0^t |f(X_s)| ds\Big] = 0\Big\}.
\end{align}
We say that the potential $V: \R^d \to \R$ belongs to \emph{$X$-Kato class}, i.e., associated with the L\'evy process
$\pro X$, whenever it satisfies
$$
\quad V_- \in \cK^X \quad \text{and} \quad V_+ \in \cK^X_{\rm loc}, \quad \text{with} \quad V_+ = \max\{V,0\}, \;
V_- = \min\{V,0\},
$$
where $V_+ \in \cK^X_{\rm loc}$ means that $V_+ 1_B \in \cK^X$, for all compact sets $B \subset \R^d$. It is straightforward
to see that $L^{\infty}_{\rm loc}(\rr^d) \subset \cK_{\rm loc}^X$, moreover, by stochastic continuity of $\pro X$ also
$\cK_{\rm loc}^X \subset L^1_{\rm loc}(\R^d)$. Note that $X$-Kato class potentials may have local singularities.

By standard arguments based on Khasminskii's Lemma, see \cite[Lem.3.37-3.38]{LHB11}, for an $X$-Kato class
potential $V$ it follows that there exist suitable constants $C_1(X,V), C_2(X,V) > 0$ such that
\begin{align}
\label{eq:khas}
\sup_{x \in \R^d} \ex^x\left[e^{-\int_0^t V(X_s)\rd s}\right] \leq \sup_{x \in \R^d}
\ex^x\left[e^{\int_0^t V_-(X_s)\rd s}\right] \leq C_1 e^{C_2t}, \quad t>0.
\end{align}
This implies that
$$
T_t f(x) = \ex^x\left[e^{-\int_0^t V(X_s) \rd s} f(X_t)\right], \quad f \in L^2(\R^d), \ t>0,
$$
are well-defined operators. Using the Markov property and stochastic continuity of $\pro X$, it can be shown that
$\{T_t: t\geq 0\}$ is a strongly continuous semigroup of symmetric operators on $L^2(\R^d)$, which we call the
\emph{Feynman-Kac semigroup} associated with the process $\pro X$ and potential $V$. In particular, by the
Hille-Yoshida theorem there exists a self-adjoint operator $H$, bounded from below, such that $e^{-t H} = T_t$, with
core $C_{\rm c}^\infty(\R^d)$. We
call the operator $H$ a \emph{non-local Schr\"odinger operator} whose kinetic term is the negative of the infinitesimal
generator $L$ of the process $\pro X$. Since any $X$-Kato class potential is relatively form bounded with respect to $-L$
with relative bound less than 1, we have
$$
H=-L +V,
$$
in form sense, and $V$ acts as a multiplication operator \cite[Ch. 3]{LHB11}. For instance, when $\pro X$ is a subordinate
Brownian motion, we have $L = \Psi(-\Delta)$, where $\Psi$ is the Laplace exponent of the corresponding subordinator (given
by a Bernstein function), see \cite{HIL12}.

We make the following standing assumption throughout this paper.
\begin{assumption}
\label{assump}
The self-adjoint operator $H = - L + V$ has a ground state, i.e., there exists an eigenfunction $\varphi_0 \in \Dom H
\subset L^2(\R^d)$ such that
\begin{equation}
H\varphi_0 = \lambda_0 \varphi_0, \quad \varphi_0 \not\equiv 0, \quad \lambda_0 = \inf \Spec H,
\label{gs}
\end{equation}
and is normalized by $\|\varphi_0\|_2 = 1$.
\end{assumption}
\begin{remark}
\rm{
By general results it follows, see \cite[Ch.3]{LHB11} and \cite{KL15,KL16}, that for the class of operators $L$
and $V$ that we consider, whenever a ground state $\varphi_0$ of $H$ does exist, it is unique, has a strictly
positive version (which will be chosen throughout below), and it is bounded and continuous, with a pointwise decay
to zero at infinity. From \eqref{eq:Lchexp}-\eqref{def:gene} we have that $\Spec (-L) = \Spec_{\rm ess} (-L) =
[0,\infty)$. It also follows by general arguments that whenever the potential is confining, i.e.,  $V(x) \to \infty$
as $|x| \to \infty$, the spectrum of $-L$ completely changes under the perturbation and the spectrum of $H$ becomes
purely discrete, consisting of isolated eigenvalues of finite multiplicities. Thus for confining potentials a ground
state (\ref{gs}) always exists. When the potential is decaying, i.e., $V(x) \to 0$ as $|x|\to \infty$, or it is
confining in one direction and decaying in another direction, the spectrum of $H$ may or may not contain a discrete
component, and the existence of a ground state depends on further details of $V$.
}
\end{remark}
\subsection{Existence and c\`adl\`ag property of ground state-transformed processes}

By using $\varphi_0 > 0$, we define the \emph{ground state transform} as the unitary map
$$
U: L^2(\R^d,\varphi_0^2\rd x) \to L^2(\R^d,\rd x), \quad f \mapsto \varphi_0 f.
$$
Also, we define the intrinsic Feynman-Kac semigroup
\begin{equation}
\label{IFK}
\widetilde{T}_t f(x) = \frac{e^{\lambda_0 t}}{\varphi_0(x)} T_t(\varphi_0 f)(x)
\end{equation}
associated with $\semi T$. Using the integral kernel $u(t,x,y)$ of $T_t$, we then have that
$\widetilde{T}_t f(x) = \int_{\R^d} \tilde u(t,x,y)f(y) \varphi_0^2(y)\rd y$, with the integral
kernel given by
\begin{equation}
\widetilde{u}(t,x,y) = \frac{e^{\lambda_0 t} u(t,x,y)} {\varphi_0(x)\varphi_0(y)},
\label{kerIFK}
\end{equation}
and infinitesimal generator $\wt L = -\wt H$, where
\begin{equation}
\label{tilH}
\widetilde H = U^{-1}(H-\lambda_0)U,
\end{equation}
with domain
$$
\Dom \widetilde H = \big\{f \in L^2(\R^d, \varphi_0^2\rd x): Uf \in \Dom H \big\}.
$$
A calculation gives the expression \eqref{eq: generator}, which holds at least for $C_{\rm c}^\infty$ functions.
The operators $\widetilde T_t = e^{t\wt L}$ are contractions and we have $\widetilde T_t 1_{\R^d} = 1_{\R^d}$ for
all $ t \geq 0$, thus $\semi {\widetilde T}$ is a Markov semigroup on $L^2(\R^d, \varphi_0^2\rd x)$.

The self-adjoint operator $\widetilde L$ generates a stationary strong Markov process, which we call \emph{ground
state-transformed (GST) process} and denote by $\pro{\widetilde X}$. GST processes have been constructed
first for Brownian motion perturbed by potentials, see \cite{S04,BL03} and \cite[Sects. 4.10.2, 4.11.9]{LHB11}
for further details and applications. However, due to the jumps in our case there are some essential
modifications, and we give a proof of the existence of a c\`adl\`ag version of GST jump processes in the
generality allowed by Assumption \ref{assump}.

Denote by $\Omega_{\rm r}$ the space of right continuous functions from $[0,\infty)$ to $\R^d$ with left
limits (i.e., c\`adl\`ag functions), and by $\Omega_{\rm l}$ the space of left continuous functions from
$[0,\infty)$ to $\R^d$ with right limits (i.e., c\`agl\`ad functions). Denote the corresponding Borel
$\sigma$-fields by $\mathcal B(\Omega_{\rm r})$ and $\mathcal B(\Omega_{\rm l})$, respectively. Also,
denote by $\Omega$ the space of c\`adl\`ag functions from $\R$ to $\R$, and its Borel $\sigma$-field by
$\mathcal B(\Omega)$.

\begin{theorem}
\label{th:exphi1}
Let $\pro X$ be a L\'evy process with generator $L$ as given by \eqref{eq:Lchexp}-\eqref{def:gene}, $V$ be
an $X$-Kato class potential, and suppose that $H=-L+V$ has a ground state
$\varphi_0$. For all $x \in \R^d$, there exists
a probability measure $\widetilde \PP^x$ on $(\Omega, {\mathcal B}(\Omega))$  and a random process $\proo
{\widetilde X}$ satisfying the following properties:
 \begin{itemize}
\item[(1)]
Let $-\infty < t_0 \leq t_1 \leq ... \leq t_n < \infty$ be an arbitrary division of the real line, for any
$n \in \N$. The initial distribution of the process is
$$
\widetilde \PP^x(\widetilde{X}_0 = x) = 1,
$$
and the finite dimensional distributions of $\widetilde \PP^x$ with respect to the stationary distribution
$\varphi_0^2 \rd x$ are given by
\begin{align}
\label{eq:fddist}
\int_{\R^d} \ex_{\widetilde \PP^x}\Big[\prod_{j=0}^n f_j(\widetilde{X}_{t_j})\Big] \varphi^2_0(x) \rd x =
\left(f_0,\: \widetilde{T}_{t_1-t_0}\: f_1 ...\: \widetilde{T}_{t_n-t_{n-1}}\: f_n \right)_
{L^2(\R^d, \varphi_0^2 \rd x)}
\end{align}
for all $f_0,f_n \in L^2(\R^d, \varphi_0^2\rd x)$, $f_j \in L^{\infty}(\R^d)$, $j=1,..., n-1$.
\item[(2)]
The finite dimensional distributions are time-shift invariant, i.e.,
$$
\int_{\R^d} \ex_{\widetilde \PP^x}\Big[\prod_{j=0}^n f_j(\widetilde{X}_{t_j})\Big] \varphi^2_0(x) \rd x =
\int_{\R^d} \ex_{\widetilde \PP^x}\Big[\prod_{j=0}^n f_j(\widetilde{X}_{t_j+s})\Big] \varphi^2_0(x) \rd x,
\quad s \in \R, \, n \in \N.
$$
\item[(3)]
$(\widetilde{X}_t)_{t \geq 0}$ and $(\widetilde{X}_t)_{t \leq 0}$ are independent, and $\widetilde X_{-t}
\stackrel{\rm d}{=} \widetilde X_t$, for all $t \in \R$.
\item[(4)]
Consider the filtrations $\left(\cF_t^{+}\right)_{t \geq 0} = \sigma\left(\widetilde{X}_s: 0 \leq s \leq t\right)$
and $\left(\cF_t^{-}\right)_{t \leq 0} = \sigma\left(\widetilde{X}_s: t \leq s \leq 0\right)$. Then
$(\widetilde{X}_t)_{t \geq 0}$ is a Markov process with respect to $\left(\cF_t^{+}\right)_{t \geq 0}$,
and $(\widetilde{X}_t)_{t \leq 0}$ is a Markov process with respect to $\left(\cF_t^{-}\right)_{t \leq 0}$.
\item[(5)]
The map $t \mapsto \widetilde X_t$ is $\widetilde \PP^x$-almost surely c\`adl\`ag.
\end{itemize}
Furthermore, we have for all $f,g\in L^{2}(\R^d,\varphi_0^2\rd x)$ the change-of-measure formula
\begin{equation}
(f, \widetilde T_t g)_{L^{2}(\R^d,\varphi_0^2\rd x)}= (f\varphi_0, e^{-t(H-\lambda_0 )}g\varphi_0)_{L^2(\R^d,\rd x)}=
\int_{\R^d}\ex_{\widetilde \PP^x}[f(\widetilde X_0) g(\widetilde X_t)]\varphi_0^2{(x)}\rd x, \quad t \geq 0.
\label{gibbs}
\end{equation}
\end{theorem}
\noindent
The probability measure $\widetilde \PP^x$ is a Gibbs measure on the space of two-sided c\`adl\`ag paths, see a
discussion for stable processes in \cite[Sect. 5.3]{KL12}. For the remaining part of this section we present a
proof of this theorem.

Let $n \in \mathbb N$ be arbitrary, and consider any time division $0 \leq t_0 \leq t_1 \leq ... \leq t_n$. Define
the set function $P_{\{t_0, ..., t_n\}} : \times_{j=0}^n \, {\mathcal B}(\R^d ) \rightarrow \R$  by
\begin{align}
&\label{fidi1}
P_{\{t_0\}}(A_0) = \left(\indiq, \widetilde T_{t_0}  \indiq_{A_0} \right)_{L^2(\R^d,\varphi_0^2\rd x)} =
\left(\indiq, \indiq_{A_0} \right)_{L^2(\R^d,\varphi_0^2\rd x)} \\
\label{fidi2}
&P_{\{t_0, ..., t_n\}}(\times_{j=0}^n A_j)= \left(\indiq_{A_0}, \widetilde T_{t_1-t_0} \indiq_{A_1} ...
\widetilde T_{t_n-t_{n-1}} \indiq_{A_n} \right)_{L^2(\R^d,\varphi_0^2\rd x)}, \quad n \in \N,
\end{align}
with arbitrary Borel sets $A_0, ..., A_n \in {\mathcal B}(\R^d )$.

\medskip
\noindent
\emph{Step 1}:
First we obtain a probability measure on the set of all functions $[0,\infty) \to \R^d$  by a projective limit of the
prescribed finite dimensional distributions (\ref{fidi1})-(\ref{fidi2}), which is a standard step. Denote the set of
finite subsets of the positive semi-axis by ${\mathcal P}_f(\R^+) = \left\{\Lambda \subset [0,\infty) : \, |\Lambda|
< \infty\right\}$, where the bars denote cardinality of the set. It can be verified directly that the family of set
functions $(P_\Lambda)_{\Lambda \in {\mathcal P}_f(\R^+)}$ satisfies the consistency condition of the marginals
$$
P_{\{t_0, ..., t_{n+m}\}}\left((\times_{j=0}^n A_j) \times (\times_{j=n+1}^{n+m} \R^d )\right) =
P_{\{t_0, ..., t_n\}}(\times_{j=0}^n A_j), \quad n, m \in \N.
$$
Hence by the Kolmogorov extension theorem there exists a probability measure $P_\infty$ and a random process $\pro Z$
on the measurable space $\left((\R^d )^{[0,\infty)}, \sigma({\mathcal A})  \right)$, where
\begin{align}
\label{bherea}
\mathcal{A} = \left\{\omega:\R\rightarrow \R^d : \; \Ran \ \omega\big|_{\Lambda}\subset E, \, E\in ({\mathcal  B} (\R^d))^{|\Lambda|},
\, \Lambda \in {\mathcal P}_f(\R^+) \right\},
\end{align}
such that
\begin{align*}
P_{\{t_0\}}(A) = \ex_{P_\infty}[\indiq_A(Z_{t_0})] \quad \mbox{and} \quad
P_{\{t_0,...,t_n\}}\left(\times_{j=0}^n A_j \right) = \ex_{P_\infty}\Big[\prod_{j=0}^n \indiq_{A_j}(Z_{t_j})\Big],
\quad n \in \mathbb N,
\end{align*}
hold. Hence we have
\begin{align}
\label{eq:fdd1}
& \ex_{P_\infty}[f_0(Z_{t_0})] = \left(1, \widetilde T_{t_0}  f_0 \right)_{L^2(\R^d,\varphi_0^2\rd x)} =
\left(\indiq, f_0 \right)_{L^2(\R^d,\varphi_0^2\rd x)} \\
\label{eq:fdd2}
& \ex_{P_\infty}\Big[\prod_{j=0}^n f_j(Z_{t_j})\Big] = \left(f_0, \widetilde T_{t_1-t_0} f_1 ...
\widetilde T_{t_n-t_{n-1}} f_n \right)_{L^2(\R^d,\varphi_0^2\rd x)},
\end{align}
for $f_j \in L^{\infty}(\R^d )$, $j=1,...,n-1$, $f_0,f_n\in L^2(\R^d,\varphi_0^2\rd x)$, $0 \leq t_0 \leq t_1
\leq ... \leq t_n$, and all $n \in \mathbb N$.

\medskip
\noindent
\emph{Step $2$}:
Next we prove the existence of both a c\`adl\`ag and a c\`agl\`ad version of $\pro Z$. In this step we show that the
Dynkin-Kinney condition holds.  The proof relies on the time and space homogeneity of the L\'evy process $\pro X$.
\begin{lemma}
\label{lm:sc}
Let $T>0$ be arbitrary but fixed. Then for every $\varepsilon > 0$ we have $P_\infty(|Z_t-Z_s|>\varepsilon) \to 0$ as
$|t-s|\to 0$, uniformly in $s,t\in[0,T]$.
\end{lemma}

\begin{proof}
We write the right hand side of \eqref{eq:fdd2} in terms of $\pro X$, i.e.,
\begin{align}
\label{eq:fdddir}
\ex_{P_\infty}\Big[\prod_{j=0}^n f_j(Z_{t_j})\Big] = \int_{\R^d} \ex^x\Big[e^{-\int_0^{t_n} (V(X_s)-\lambda_0) \rd s}
\Big(\prod_{j=0}^n f_j(X_{t_j})\Big) \varphi_0 (X_{t_n})\Big] \varphi_0 (x) \rd x.
\end{align}
Let $0 \leq s < t
\leq T$, and denote by $B_\varepsilon(x)$ the ball of radius $\varepsilon$ centered in $x$. 
Notice that by \eqref{eq:fdd2} we have
$$
P_\infty(|Z_t-Z_s|>\varepsilon) = \left(\indiq, \widetilde T_{t-s} \indiq_{B_\varepsilon(0)^c} \right)_{L^2(\R^d,\varphi_0^2\rd x)}.
$$
Recall that $\ex^x$ is expectation with respect to the measure $\PP^x$ of the L\'evy process $\pro X$. By \eqref{eq:fdddir},
{the Markov property of $\pro X$, and the conservative property of the intrinsic semigroup $\semi{\wt T}$ on
$L^2(\RR^d, \ph_0^2\rd x)$}, we furthermore have
\begin{align*}
P_\infty(|Z_t-Z_s|>\varepsilon) =
\int_{\R^d} \ex^x\left[e^{-\int_0^{t-s} (V(X_r) - \lambda_0) \rd r}  \varphi_0 (X_{t-s})
\indiq_{B_\varepsilon(x)^c}(X_{t-s})\right] \varphi_0 (x) \rd x.
\end{align*}
Schwarz inequality gives
\begin{align}
&P_\infty(|Z_t-Z_s|>\varepsilon)\nonumber \\
& \label{stochcon}
\leq
\int_{\R^d} \left(\ex^x\left[\varphi_0^2 (X_{t-s})\right]\right)^{1/2}
\left(\ex^x\left[\indiq_{B_\varepsilon(x)^c}(X_{t-s})e^{-2\int_0^{t-s} (V(X_r) - \lambda_0) \rd r}\right]\right)^{1/2}
\varphi_0 (x)\rd x.
\end{align}
Using again Schwarz inequality, we have
\begin{eqnarray*}
\lefteqn{
\ex^x\left[\indiq_{B_\varepsilon(x)^c}(X_{t-s})e^{-2\int_0^{t-s} (V(X_r) - \lambda_0) \rd r}\right]}\\
&\quad \leq&
\left(\ex^x\left[\indiq_{B_\varepsilon(x)^c}(X_{t-s}) \right]\right)^{1/2}
\left( \ex^x\left[e^{-4\int_0^{t-s} (V(X_r) - \lambda_0) \rd r}\right]\right)^{1/2}
\leq C \,\PP(|X_{t-s}|>\varepsilon)^{1/2},
\end{eqnarray*}
where $C=\sup_{x\in\R^d} \left( \ex^x\left[e^{-4\int_0^{t-s} (V(X_r) - \lambda_0) \rd r}\right]\right)^{1/2}$.
Thus by \eqref{stochcon} and a repeated use of Schwarz inequality, we get
\begin{align*}
P_\infty(|Z_t-Z_s|>\varepsilon)
\leq
C^{1/2} \PP(|X_{t-s}|>\varepsilon)^{1/4}   \|\varphi_0\|_2^2,
\end{align*}
which goes to zero as $|t-s| \to 0$ by stochastic continuity of $\pro X$.

\enlargethispage{0.5cm}
To show uniform convergence, fix $\eta>0$. By stochastic continuity, for every $t$ there exists $\rho_t>0$ such
that $P_\infty(|Z_t-Z_s|\geq\frac{\varepsilon}{2})\leq \frac{\eta}{2}$ for $|s-t|<\rho_t$. Let $I_t=
(t-\frac{\rho_t}{2},t+\frac{\rho_t}{2})$. There is a finite covering $I_{t_j}$, $j=1,...,n$, such that
$\cup_{j=1}^n I_{t_j}\supset [0,T]$. Let $\rho = \min_{1\leq j\leq n}\rho_{t_j}$. If $|s-t|<\rho$ and
$s,t\in[0,T]$, then $t\in I_{t_j}$ for some $j$, hence $|s-t_j|<\rho_{t_j}$ and
$$
P_\infty(|Z_t-Z_s|>\varepsilon) \leq P_\infty(|Z_t-Z_{t_j}|>\varepsilon/2)+ P_\infty(|Z_s-Z_{t_j}|>\varepsilon/2)
<\eta.
$$
\end{proof}
Recall that $\Omega_{\rm r}$ and $\Omega_{\rm l}$ denote the c\`adl\`ag and c\`agl\`ad path spaces over $[0,\infty)$,
respectively. By Lemma \ref{lm:sc}, there exists a c\`adl\`ag version $\bar Z = \pro {\bar Z}$ of $\pro Z$ on the space
$((\R^d )^{[0,\infty)}, \sigma({\mathcal  A}), P_\infty)$. Denote the image measure of $P_\infty$ on
$(\Omega_{\rm r}, {\mathcal B}(\Omega_{\rm r}))$ by $Q_{\rm r}  = P_\infty \circ \bar Z^{-1}$. Let  $\pro {Y}$
be the coordinate process on $(\Omega_{\rm r} , {\mathcal B}(\Omega_{\rm r}), Q_{\rm r})$ such that $\bar Z_t
\stackrel{\rm d}{=} Y_t$. In terms of $\pro {Y}$, equalities \eqref{eq:fdd1}-\eqref{eq:fdd2} become
\begin{align}
& \label{mmaa}
\left(\indiq, \widetilde T_{t_0} f_0 \right)_{L^2(\R^d,\varphi_0^2\rd x)} =
\left(\indiq, f_0 \right)_{L^2(\R^d,\varphi_0^2\rd x)} = \ex_{Q_{\rm r}}[f_0(Y_{t_0})], \\
& \label{mma}
\left(f_0, \widetilde T_{t_1-t_0} f_1 ... f_{n-1}\widetilde T_{t_n-t_{n-1}} f_n \right)_{L^2(\R^d,\varphi_0^2\rd x)}
= \ex_{Q_{\rm r}}\Big[\prod_{j=0}^n f_j(Y_{t_n})\Big].
\end{align}
Similarly, a c\`agl\`ad version $\underline Z = \pro {\underline Z}$ of $\pro Z$ can be also constructed on the same
probability space $((\R^d )^{[0,\infty)}, \sigma({\mathcal  A}), P_\infty)$. Likewise, there exists a probability measure
$Q_{\rm l}$ on the c\`agl\`ad space $(\Omega_{\rm l}, \mathcal B(\Omega_{\rm l})$ such that the coordinate process
$\pro{Y}$ satisfies \eqref{mmaa}-\eqref{mma} with $Q_{\rm l}$, and $Q_{\rm l} = P_\infty\circ \underline Z^{-1}$
holds.

\medskip
\noindent
\emph{Step 3}:
We define the regular conditional probability measures $Q_{\rm r}^x (\,\cdot\,) = Q_{\rm r}(\,\cdot\,|Y_0=x)$
and $Q_{\rm l}^x (\,\cdot\,) = Q_{\rm l}(\,\cdot\,|Y_0=x)$ for $ x \in \R^d$ on $(\Omega_{\rm r}, \mathcal
B(\Omega_{\rm r}))$ and $(\Omega_{\rm l}, {\mathcal  B}(\Omega_{\rm l}))$, respectively. Since $Y_0$ is
distributed by $\varphi_0^2(x)\rd x$, we have $Q_{\rm r}(A)=\int_{\R^d}\ex_{Q_{\rm r}^x}[\indiq_A] \varphi_0^2(x)\rd x$ and
$Q_{\rm l}(A) = \int_{\R^d} \ex_{Q_{\rm l}^x}[\indiq_A] \varphi_0^2(x)\rd x$. Hence the process $(Y_t)_{t \geq 0}$
on $(\Omega_{\rm r}, {\mathcal  B}(\Omega_{\rm r}), Q_{\rm r}^x)$ satisfies
\begin{align}
\label{eq:fdd4}
& \left(\indiq, \widetilde T_{t_0}  f_0 \right)_{L^2(\R^d,\varphi_0^2\rd x)} =
\left(\indiq, f_0 \right)_{L^2(\R^d,\varphi_0^2\rd x)}
= \int_{\R^d}\ex_{Q_{\rm r}^x}[f_0(Y_{t_0})] \varphi_0^2(x)\rd x \\
\label{eq:fdd3}
& \left(f_0, \widetilde T_{t_1-t_0} f_1 ... f_{n-1}\widetilde T_{t_n-t_{n-1}} f_n \right)_{L^2(\R^d,\varphi_0^2\rd x)}
= \int_{\R^d}  \ex_{Q_{\rm r}^x}\Big[\prod_{j=0}^n f_j(Y_{t_j})\Big] \varphi_0^2(x)\rd x,
\end{align}
and the process $(Y_t)_{t \geq 0}$ on $(\Omega_{\rm l}, {\mathcal  B}(\Omega_{\rm l}), Q_{\rm l}^x)$ satisfies
\eqref{eq:fdd4}-\eqref{eq:fdd3} with $Q_{\rm l}^x$.

\begin{lemma}
The coordinate process $\pro {Y}$ is a Markov process on $(\Omega_{\rm r},{\mathcal B}(\Omega_{\rm r}), Q_{\rm r}^x)$
with respect to the natural filtration $\pro {{\mathcal F}}$. Similarly, the coordinate process $\pro {Y}$ is a
Markov process on the probability space $(\Omega_{\rm l}, {\mathcal  B}(\Omega_{\rm l}), Q_{\rm l}^x)$ with respect to the
natural filtration
$\pro {{\mathcal F}}$.
\end{lemma}
\begin{proof}
Let $q_t(x,A) = \widetilde T_t \indiq_A(x)$, for every $A \in {\mathcal  B}(\R^d )$, $x \in \R^d $ and $t\geq 0$. Clearly,
$q_t(x,A) = \ex_{Q_{\rm r}^x}[\indiq_A(Y_t)]$, and by \eqref{eq:fdd4}-\eqref{eq:fdd3} the finite dimensional
distributions of $(Y_t)_{t \geq 0}$ can be written as
\begin{align}
\label{eq:fdd5}
\ex_{Q_{\rm r}^x} \Big[\prod_{j=0}^n \indiq_{A_j} (Y_{t_j})\Big] =
\int_{\R^d} \prod_{j=0}^n \indiq_{A_j}(x_j) q_{t_j - t_{j-1}}(x_{j-1},\rd x_j),
\end{align}
with $t_0 = 0$ and $ x_0 = x$. By using the properties of the semigroup $\semi {\widetilde T}$, it is checked directly
that $q_t(x,A)$ is a probability transition kernel, thus $(Y_t)_{t \geq 0}$ is a Markov process with finite dimensional
distributions given by \eqref{eq:fdd5}. The second statement can be proven similarly.
\end{proof}

\noindent
\emph{Step 4}:
Next we construct a random process indexed by the whole real line $\R$. Consider $\widehat {\Omega}= \Omega_{\rm r}
\times \Omega_{\rm l}$, with product $\sigma$-field $\widehat {{{\mathcal F}}}={\mathcal  B}(\Omega_{\rm r} ) \times
{\mathcal  B}(\Omega_{\rm l})$ and product measure ${\widehat Q}^x = Q_{\rm r}^x\times Q_{\rm l}^x$. Define the
coordinate process $\pro{\widehat Y}$ by
$$
\widehat  Y_t(\omega)=
\begin{cases}
\omega_1(t)  & t\geq0\\
\omega_2(-t) & t<0
\end{cases}
$$
for $\omega =(\omega_1, \omega_2) \in \widehat {\Omega}$. This is then a random process $(\widehat Y_t)_{t \in \R}$
on $(\widehat {\Omega}, \widehat {{{\mathcal F}}},{\widehat Q}^x)$ such that ${\widehat Q}^x(\widehat Y_0 = x)=1$,
and $\R \ni t \mapsto \widehat Y_t(\omega)$ is c\`adl\`ag. It is direct to see that $\widehat Y_t$, $t \geq 0$, and
$\widehat Y_s$, $s \leq 0$, are independent, and $\widehat Y_t \stackrel {\rm d}{=} \widehat Y_{-t}$.

\medskip
\noindent
\emph{Step 5}:
Denote the image measure of ${\widehat Q}^x$ on $(\Omega_{\rm r}, {\mathcal B}(\Omega_{\rm r}))$ with respect to
$\widehat Y = (\widehat Y_t)_{t \in \R}$ by $\widetilde \PP^x = {\widehat Q}^x \circ \widehat Y^{-1}$. Let
$\widetilde X_t(\omega) = \omega(t)$, $t \in \R$, $\omega \in \Omega$, denote the coordinate process. Clearly, we
have $\widetilde X_t \stackrel{\rm d}{=} Y_t$ for $t\in\R$. Thus we see that $\widetilde X_t \stackrel{\rm d}{=}
\widetilde X_{-t}$, and by Step 4 above $(\widetilde X_t)_{t \geq 0}$ and $(\widetilde X_t)_{t \leq 0}$ are independent.
Furthermore, by Step 2 we have that $(\widetilde X_t)_{t \geq 0}$ and $(\widetilde X_{t})_{t \leq 0}$ are Markov processes
with respect to $({{\mathcal F}}^{+}_t)_{t \geq 0}$ and $({{\mathcal F}}^{-}_t)_{t \leq 0}$, respectively.

To prove shift invariance, consider arbitrary time-points $t_0\leq\ldots \leq t_n\leq 0\leq t_{n+1}\leq\ldots \leq t_{n+m}$,
$n, m \in \N$. Then by independence of $(\widetilde X_t)_{t\leq 0}$ and $\pro {\widetilde X}$ we have
$$
\int_{\R^d} \ex_{\widetilde \PP^x}\Bigl[\prod_{j=0}^{n+m}f_j(\widetilde X_{t_j})\Bigr] \varphi_0^2\rd x
= \int_{\R^d} \ex_{\widetilde \PP^x} \Bigl[\prod_{j=0}^{n}f_j(\widetilde X_{t_j})\Bigr]
\ex_{\widetilde \PP^x}\Bigl[\prod_{j=n+1}^{n+m}f_j(\widetilde X_{t_j})\Bigr] \varphi_0^2\rd x.
$$
Moreover,
$$
\ex_{\widetilde\PP^x} \Big[\prod_{j=n+1}^{n+m} f_j(\widetilde X_{t_j})\Big] = (\widetilde T_{t_{n+1}} f_{n+1}
\widetilde T_{t_{n+2}-t_{n+1}} f_{n+2} \ldots \widetilde T_{t_{n+m}-t_{n+m-1}}f_{n+m})(x)
$$
and
$$
\ex_{\widetilde\PP^x}\Big[\prod_{j=0}^{n}f_j(\widetilde X_{t_j})\Big]
= \ex_{\widetilde \PP^x}\Big[\prod_{j=0}^{n}f_j(\widetilde X_{-t_j})\Big] =
(\widetilde T_{-t_n} f_{n} \widetilde T_{t_{n}-t_{n-1}} f_{n-1}\ldots \widetilde T_{t_{1}-t_{0}}f_0 )(x).
$$
A combination of the above gives
\begin{eqnarray*}
\int_{\R^d}\ex_{\widetilde\PP^x}\Big[\prod_{j=0}^{n+m}f_j(\widetilde X_{t_j})\Big]\varphi_0^2(x)\rd x
&=&
(\widetilde T_{-t_{n}} f_{n} \ldots \widetilde T_{t_{1}-t_{0}}f_0, \widetilde T_{t_{n+1}} f_{n+1} \ldots
\widetilde T_{t_{n+m}-t_{n+m-1}}f_{n+m})_{L^2(\R^d,\varphi_0^2\rd x)}\\
&=&
(f_0, \widetilde T_{t_1-t_0} f_1 \ldots \widetilde T_{t_{n+m}-t_{n+m-1}}f_{n+m})_{L^2(\R^d,\varphi_0^2\rd x)}.
\end{eqnarray*}
This implies the required time-shift invariance. Formula (\ref{gibbs}) is now a direct consequence, and this
completes the proof of the theorem.

\section{Local path regularity of GST processes}
\subsection{Stochastic differential equation with jumps associated with the GST process }
The generator $\wt L = - \wt H$ of the ground state-transformed process $\pro {\widetilde X}$ can be determined explicitly
yielding \eqref{eq: generator}, which has first appeared in \cite{KL16b}. Since $H$ and $\wt H$ are unitary equivalent
 by \eqref{tilH}, we have $\Dom(\wt H) = U \Dom(H)$, and since $H$ is closed, also $\wt H$ is a closed operator.
 Moreover, since $H$ is self-adjoint, $\wt H$ is also self-adjoint with core $C_{\rm c}^\infty(\R^d)$.

{To study the multifractal spectrum, first we show the existence of a solution to the martingale problem for
$(\wt L, C_{\rm c}^2(\RR^d))$, and provide a jump SDE representation for the ground state-transformed process},
which we call the \emph{ground state SDE}. We write $\rr^d_* = \rr^d \setminus \{0\}$ for a shorthand notation.

For the ground state SDE we will use the following condition (see also Remark \ref{gsreg} below).
\begin{assumption}
\label{assumpdrift}
Let $\varphi_0$ be the ground state of $H$. We assume that the function $x \mapsto \nabla \ln \varphi_0(x)$,
$x \in \R^d$, is locally bounded.
\end{assumption}

\begin{theorem}\label{uniqueness}
Let Assumptions \ref{assump} and \ref{assumpdrift} hold. Consider the stochastic differential equation with jumps
\begin{multline}
\label{sde}
M_t = M_0 + {\sigma} B_t + \int_0^t \sigma\nabla\ln \ph_0(M_s)\,\rd s + \int_0^t\int_{|z|\le 1}
\frac{\ph_0(M_{s}+z)-\ph_0(M_{s})}{\ph_0(M_{s})} z \nu(z) \rd z \rd s  \\
+\int_0^t\int_{|z|\le 1} \int_0^{\infty} z \indiq_{\left\{v\le \frac{\ph_0(M_{s-}+z)}{\ph_0(M_{s-})}\right\}}
\wt N(\rd s, \rd z, \rd v)
+\int_0^t \int_{|z|>1} \int_0^{\infty} z\indiq_{\bra{v\le \frac{\ph_0(M_{s-}+z)}{\ph_0(M_{s-})}}}
N(\rd s, \rd z, \rd v),
\end{multline}
where $\pro B$ is an {$\RR^d$}-valued Brownian motion with covariance matrix $\sigma$, and $N$ is a Poisson random
measure on $[0,\infty)\times\rr^d_*\times[0,\infty)$ with intensity $\rd t\nu(z)\rd z \rd v$.
The GST process constructed in Theorem \ref{th:exphi1}  is a weak solution of \eqref{sde}.
\end{theorem}
\begin{proof}
Let
\begin{equation}
\label{mt}
\widetilde X^f_t = f(\widetilde X_t) - f(\widetilde X_0) - \int_0^t\widetilde L f(\widetilde X_r)\rd r, \quad t\geq 0,
\end{equation}
where $f \in \Dom(\widetilde L)$. Using a general result, see e.g. Kurtz \cite[Th. 2.3]{K10}, we have that
$(\widetilde X, \widetilde\PP^x)$ is a weak solution of the SDE \eqref{sde} if and only if $\widetilde\PP^x$ solves
the $(\widetilde L, C^2_{\rm c})$ martingale problem with initial value $x\in\RR^d$, that is, $\pro {\wt X^f}$ is a
martingale under $\wt\PP^x$, for all $f\in C^2_{\rm c}(\RR^d)$.

Using Assumption \ref{assumpdrift} and that $\varphi_0 > 0$ is bounded continuous, we see that the functions $x
\mapsto \int_{\R^d} \varphi_0(z+x) / \varphi_0(z)(1 \wedge |z|^2) \nu(\rd z)$, $x \mapsto \int_{|z| \leq 1}  z
(\varphi_0(z+x) - \varphi_0(z))/\varphi_0(z) \nu(\rd z)$ and $x \mapsto \nabla \log \varphi_0(x)$ are locally bounded,
and the conditions in \cite[Th. 2.3]{K10} hold. Also, since $\wt L$ is a closed operator, by using a mollifier we
can close $C^\infty_{\rm c}(\R^d)$ in the $C^2$-norm as in \cite[Th. 2.37]{BSW} to obtain that $C^2_{\rm c}(\R^d)
\subset \Dom(\wt L)$.

Let $(M, P^x)$ be a weak solution to \eqref{sde} on a suitable probability space with probability measure $P^x$, and
starting point $P^x(M_0=x)=1$. Write for the drift $b: \RR^d\to \RR^d$,
\begin{align*}
b(x) = \sigma\nabla\ln \ph_0(x) + \int_{|z|\le 1}\frac{\ph_0(x+z)-\ph_0(x)}{\ph_0(x)} z\nu(z)\rd z.
\end{align*}
Using It\^o's formula for $\RR^d$-valued semimartingales, see e.g. \cite{JS03}, for every real-valued $f\in C^2_{\rm c}(\RR^d)$
we have
\begin{eqnarray*}
 f(M_t)
 &=&
 f(M_0) + \int_0^t \nabla f(M_{s-})\cdot\rd (\sigma B_s) + \int_0^t \nabla f(M_{s})\cdot b(M_{s}) \rd s +
 {{1\over 2} \int_0^t \sigma \nabla \cdot \sigma \nabla f(M_s) \rd s} \\
 &&
 + \int_0^t\int_{\RR^d_*}\int_0^\infty \left[f\left(M_{s-}+
 z\indiq_{\bra{v\le \frac{\ph_0(M_{s-}+z)}{\ph_0(M_{s-})}}}\right) - f(M_{s-}) \right]\wt N(\rd s, \rd z, \rd v) \\
 &&
 + \int_0^t\int_{|z|> 1}\int_0^\infty \left[f\left(M_{s-}+ z\indiq_{\bra{v\le \frac{\ph_0(M_{s-}+z)}{\ph_0(M_{s-})}}}\right)
 - f(M_{s-})\right] \rd v\nu(z)\rd z\rd s \\
 &&
 +  \int_0^t\int_{|z|\le 1}\int_0^\infty \left[f\left(M_{s-}+
 z\indiq_{\bra{v\le \frac{\ph_0(M_{s-}+z)}{\ph_0(M_{s-})}}}\right) - f(M_{s-})
 \right.
 \\
 && \qquad \qquad \qquad \qquad \qquad \qquad \left.
  - \nabla f(M_{s-})\cdot z\indiq_{\bra{v\le \frac{\ph_0(M_{s-}+z)}{\ph_0(M_{s-})}}} \right] \rd v\nu(z)\rd z\rd s.
\end{eqnarray*}
Note that in the second-to-last integral the integrand is zero for all $v$ larger than the ground state ratio
$\ph_0(M_{s-}+z)/\ph_0(M_{s-})$. It is thus equal to
\begin{align*}
\int_0^t\int_{|z|>1} \big( f(M_{s-}+z) - f(M_{s-})\big) \frac{\ph_0(M_{s-}+z)}{\ph_0(M_{s-})} \nu(z)\rd z \rd s.
\end{align*}
Similarly, the last integral equals
\begin{align*}
\int_0^t\int_{|z|\le 1} \big( f(M_{s-}+z) - f(M_{s-}) - z \cdot \nabla f(M_{s-}) \big) \frac{\ph_0(M_{s-}+z)}{\ph_0(M_{s-})}
\nu(z)\rd z \rd s.
\end{align*}
Since $f$ has bounded first and second derivatives, the Brownian component and the compensated Poisson integral are
martingales, therefore
\begin{align*}
M^f_t = f(M_t) - f(M_0) - \int_0^t \wt L f(M_s) \rd s, \quad t \geq 0,
\end{align*}
is a $P^x$-martingale.

It remains to show that the probability measures $P^x = \widetilde\PP^x$ constructed in Theorem \ref{th:exphi1}
solve the $(\widetilde L, C^2_{\rm c})$ martingale problem with initial value $x\in\RR^d$. Consider the natural
filtration $\pro {\mathcal F}$ of $\pro {\widetilde X^f}$ and let $0 \leq s \leq t$. Since $\ex_{\widetilde \PP^x}
[\widetilde X^f_t|\mathcal{F}_s] = \wt X^f_s + \ex_{\widetilde \PP^x}[\wt X^f_t-\wt X^f_s|\mathcal{F}_s]$, we
only need to show that the second term vanishes. By the Markov property of $\pro {\widetilde X}$ established in
Theorem \ref{th:exphi1}, we have
$$
\ex_{\widetilde \PP^x}[f(\widetilde X_t)|\mathcal{F}_{s}] = \widetilde T_{t-s}f(\widetilde X_s),
\quad 0 \leq s \leq  t.
$$
By differentiability of the function $t\mapsto \widetilde T_t$, we obtain for all $t\geq 0$ that
$\frac{d}{dt}\widetilde T_tf=  \widetilde L \widetilde T_{t} f = \widetilde T_t \widetilde L f$, and hence
$\widetilde T_{t}f - f = \int_0^t \widetilde L \widetilde T_r f \rd r$.
Thus we have
\begin{eqnarray*}
\ex_{\widetilde \PP^x}\left[ f(\widetilde X_t) - f(\widetilde X_s) -
\int_s^t \widetilde L f(\widetilde X_r)\rd r\left| \mathcal{F}_s\right.\right]
&=&
\widetilde T_{t-s} f(\widetilde X_s)-f(\widetilde X_s) - \int_s^t \widetilde L \widetilde T_{r-s} f (\widetilde X_s)\rd r \\
&=&
\widetilde T_{t-s}f(\widetilde X_s) - f(\widetilde X_s) - \int_0^{t-s} \widetilde L \widetilde T_r f (\widetilde X_s)\rd r \\
&=&
\widetilde T_{t-s}f(\widetilde X_s) - f(\widetilde X_s) - \widetilde T_{t-s}f(\widetilde X_s) + f(\widetilde X_s) =0,
\end{eqnarray*}
as required.
\end{proof}

\begin{remark}
\label{gsreg}
\hspace{100cm}
\begin{trivlist}
\rm{
\item[\, (1)]
In general, little information is available on the regularity of $\varphi_0$. In some specific cases of potentials
growing to infinity at infinity and the operator $L = (-\rd^2/\rd x^2)^{1/2}$, it is known that the ground state is
analytic \cite{LM12,DL16}. However, it is also known that the ground state for Brownian motion in a finitely deep
potential well, i.e., $V(x) = - v \indiq_{\{|x| \leq a\}}$, $v, a > 0$, is only $C^1$.
\item[\, (2)]
The conditions under which the ground state SDE has a solution and Theorem \ref{uniqueness} holds can be improved.
For a large class of L\'evy processes $\pro X$ and potentials $V$, it can be shown that for large enough $|x|$ and
suitable constants $C_1, C_2 > 0$,
\begin{eqnarray*}
&& C_1 \frac{\nu(x)}{V(x)} \leq \varphi_0(x) \leq C_2 \frac{\nu(x)}{V(x)} \quad
\mbox{for $V(x) \to \infty$ \ as \ $|x|\to\infty$} \\
&& \\
&& C_1 \nu(x) \leq \varphi_0(x) \leq C_2 \nu(x) \qquad \mbox{for $V(x) \to 0$ \ as \ $|x|\to\infty$}.
\end{eqnarray*}
For precise statements and conditions we refer to \cite{KL15,KL16}. For illustration consider $d=1$, $V(x)
\asymp x^{2m}$, $m > 0$, and a symmetric $\alpha$-stable process; then the above implies $\varphi_0(x) \asymp
|x|^{-d-\alpha-2m}$, and thus far enough from the origin the drift would become $b(x) = \sigma\frac{\rd}{\rd x}
\ln\varphi_0(x) \asymp - \frac{\sgn(x)}{|x|}$. Hence for large enough values of $X_t$ we get $X_t b(X_t) < 0$,
and this pull-back mechanism would prevent the paths from exploding. Since our main concern here is the
multifractal behaviour of GST processes, the ground state SDE will be studied in further detail elsewhere.
\item[\, (3)]
We note that we are not concerned with uniqueness of the solution of the martingale problem. Since we show below
that any solution of the SDE \eqref{sde} has the same multifractal nature, we only need to know that the GST
process is a solution.
}
\end{trivlist}
\end{remark}

\subsection{Multifractal spectrum of GST processes}
Now we are in the position to state and prove the multifractal nature of local H\"older exponents of ground
state-transformed processes via the ground state SDE. Recall the notations in \eqref{eq: spec with B}-\eqref{eq: spec pure jump}.

\begin{theorem}
\label{spectrum}
Let Assumptions \ref{assump} and \ref{assumpdrift} hold, and 
$(M, \PP^x)$ be a weak solution of \eqref{sde}.
\begin{itemize}
\item[(1)]
If {$\sigma\neq 0$}, i.e., the underlying L\'evy process has a Brownian component, then almost surely
\begin{align*}
D_M(h)= D^1_X(h), \quad h > 0.
\end{align*}
\item[(2)]
If $\sigma=0$, i.e., the underlying L\'evy process is a pure jump process, and
\vspace{0.1cm}
\begin{enumerate}
\item[(i)]
either $\be_\nu\in[1,2]$,
\vspace{0.1cm}
\item[(ii)]
or $\be_\nu\in(0,1)$ with $\varphi_0\in C^{k+1}(\RR^d)$ and  $k\ge 1/\be_\nu$,
\end{enumerate}
\vspace{0.1cm}
then almost surely
\begin{align*}
D_M(h) = D^2_X(h), \quad h > 0.
\end{align*}
\end{itemize}
\end{theorem}
\noindent
In the remainder of this section we prove this theorem through a sequence of auxiliary results.

\subsubsection{Associated Poisson point process}
Recall that the jumps of the Poisson measure $N$ give rise to a Poisson point process with  measure
$\nu(z)\rd z$. The pointwise regularity of L\'evy processes with infinite jump measure, i.e., $\int_{\RR^d}\nu(z)\rd z=
\infty$, relies on a configuration of a dense set of jumps. For the process $\pro M$ the jump configuration is more
involved as it depends on the entire paths of the process. Indeed, the factor featuring the indicator function in the 
compensated Poisson integral in \eqref{sde} implies that our process jumps only as long as the ground state ratio is 
not too small. Thus $\pro M$ jumps less often than the underlying L\'evy process. We prove that the underlying Poisson 
point process characterizes the pointwise regularity of $\pro M$.

Let $N(\rd t,  \rd z, \rd v)$ be a Poisson measure with intensity $\rd t\, n(\rd z, \rd v)$ on $\RR^+\times E$ with
$E=\RR^d_*\times (0,\infty)$, endowed with the product Borel $\sigma$-field ${\mathcal B}(E)$. Let $\{E_k, k\in\NN_*\}$
be a partition of $E$, with $E_k\in\mathcal B(E)$ and $n(E_k)<\infty$.  It is well-known \cite[Ths 8.1, 9.1]{IW} that
there exists
\begin{itemize}
\item a sequence of  exponential random variables  $\{\tau^{(k)}_i, i\in\NN\}$ with parameter $n(E_k)$, \sk
\item a sequence of random variables $\{\xi^{(k)}_i, i\in\NN\}$ with distribution $\indiq_{E_k}n(\rd z,\rd v)/n(E_k)$, \sk
\end{itemize}
such that
\begin{align*}
 N((0,t]\times U) = |\bra{s\in D: s\le t, \, p(s)\in U}|, \; \mbox{ for all } t>0, \, U \in \mathcal B(E),
\end{align*}
where $p$ is the point process defined by
\begin{align*}
  p\( \sum_{\ell=0}^i \tau^{(k)}_\ell\) = \xi^{(k)}_i, \quad k, i=1,2,\ldots
 \end{align*}
and
\begin{align*}
D = \bigcup_{k=0}^\infty \bra{\sum_{\ell=0}^i \tau^{(k)}_\ell: i\in\NN }.
\end{align*}
Here, all $\tau^{(k)}_i, \xi^{(k)}_i$ are mutually independent random variables on the same probability space.  Extending
the probability space, if necessary, by passing to a product probability space, we can find a sequence of uniform random
variables $\eta_i$ in $[0,1]$ that is independent of $p$. Define $p': D\to E\times [0,1]$ by
\begin{align*}
p'\( \sum_{\ell=0}^i \tau^{(k)}_\ell\) = (\xi^{(k)}_i,\eta_i) \quad k,i=1,2,\ldots
\end{align*}
It can be shown \cite[Ths 8.1, 9.1]{IW} that the counting measure
\begin{align*}
 N_{p'}((0,t]\times U\times I) = \left|\bra{s\in D: s\le t, \, p'(s)\in U\times I}\right|, \; \mbox{ for all }
 t>0, \, U \in \mathcal B(E),  \, I\in\mathcal B([0,1]),
\end{align*}
is also a Poisson measure, with intensity  $\rd t\, n(\rd z, \rd v)\indiq_{[0,1]}(x)\rd x$. In particular, almost surely,
\begin{align}\label{eq:N=N_p'}
N((0,t]\times U) = N_{p'}((0,t]\times U\times [0,1]).
\end{align}
From now on, we consider the Poisson measure $N$ as part of the weak solution
$(M, \widetilde \PP^x)$ of the SDE \eqref{sde} on a probability space. Possibly on the extended probability space, we have
$N_{p'}$ satisfying \eqref{eq:N=N_p'} which will serve as an auxiliary measure to prove a covering property satisfied by a
family of point systems induced by the jumps of our process $\pro M$. Informally, the measure $N_{p'}$ allows us to remove
spatial dependence of the jump kernel; a similar argument has been first used in \cite{X15}.  In what follows we will write
\begin{align*}
p = \{(s,z(s), v(s)): \, s\in  D\} \quad \mbox{and} \quad  p' = \{(s,z(s), v(s), x(s)): \, s\in D \}.
\end{align*}

\subsubsection{H\"older regularity}\label{sec:holder}

First we determine the pointwise H\"older exponent of the sample paths of the ground state SDE under the assumption
that the ratio of ground state evaluations appearing in the coefficients of the SDE is bounded both from below and above,
i.e., we assume that there exists $0<c<1$ such that
\begin{align}
\label{eq: ratio control} c \le \frac{\ph_0(x+z)}{\ph_0(x)} \le 1/c,    \quad x\in \rr^d, |z|\le 1.
\end{align}
In a next step we remove this constraint by using a localization argument to get the result in a desirable generality.

The following general result is due to Jaffard \cite[Lem. 1]{J99}, which is essential in deriving an upper bound for
the H\"older exponent of a locally bounded function with a dense set of jump discontinuities.

\begin{lemma}
\label{jumplemma}
Let $f:\rr \to \rr^d$ be a c\`adl\`ag function having a dense set of jump discontinuities of size $z_n$ at the time-points
$t_n$. Then, for every $t\in \rr$ and every sequence of jump discontinuities $t_{n_k} \to t$ as $k \to \infty$, we have
\begin{align*}
H_f(t) \le \liminf_{k\to \infty} \frac{\ln z_{n_k}}{\ln |t-t_{n_k}|}.
\end{align*}
\end{lemma}

Clearly, only the small jumps have an impact on the local regularity. Write
\begin{align*}
J = \bra{ s\ge 0:  |z(s)|\le 1,  \, v(s)\le  \frac{\ph_0(M_{s-}+z(s))}{\ph_0(M_{s-})} }.
\end{align*}
By the properties of the (compensated) Poisson integral, the solution to the ground state SDE \eqref{sde} makes
a jump at each $s\in J$, of size $|z(s)|$. Borrowing an idea from \cite{J99}, we consider a family of
limsup sets built from the Poisson point process $p$. Recall that $\be_\nu$ is the Blumenthal-Getoor
index of the L\'evy measure $\nu(z)\rd z$ defined in \eqref{eq: BG index}.
 For all $\de>0$, define
\begin{align*}
A(\e, \de) = \bigcup_{s\in J, \  |z(s)|\le \e} (s-|z(s)|^{\be_\nu\de}, s+ |z(s)|^{\be_\nu\de}),
\end{align*}
and
\begin{align}\label{eq: A_de}
A(\de) = \limsup_{\e\downarrow 0} A(\e, \de).
\end{align}
This family of sets satisfies a convenient covering property when $\de<1$.
\begin{lemma}\label{lem: covering subcritical}
For all $\de<1$ we have $A_\de= [0,\infty)$, almost surely.
\end{lemma}
\begin{proof}
Define
\begin{align*}
J' = \bra{ s\ge 0:  |z(s)|\le 1,  \, v(s)\le  \frac{\ph_0(M_{s-}+z(s))}{\ph_0(M_{s-})}, \,
x(s) \le \frac{c \ph_0(M_{s-})}{\ph_0(M_{s-}+z(s))} }  \subset J,
\end{align*}
and $A'_\de$ as in \eqref{eq: A_de} with $J$ replaced by $J'$. Observe that by the lower bound in \eqref{eq: ratio control},
the right hand side of the bound concerning $x(s)$ in the above set is a random number in $[0,1]$. Since $A'_\de\subset A_\de$
for all $\de\ge 0$,  it remains to show that for any fixed $\de<1$, we have $A'_\de = [0,\infty)$ almost surely. The result
then follows by the monotonicity of the sets $A_\de$ in $\de$.

\medskip
\noindent
\emph{Step 1}:  First we note that the counting measure
\begin{align*}
 \mu(\rd s, \rd y) =\sum_{s\in J}  \delta_{(s, |z_s|^{\de\be_\nu})}
\end{align*}
is a Poisson random measure with intensity $\rd s (c\pi_\de(\rd y))$ on $\rr^+\times(0,1]$, where $\pi_\de$ is the image
measure of $\nu(z)\rd z 1_{|z|\le 1}$ by the map $z\mapsto |z|^{\de\be_\nu}$ and  $c$ is the constant in
\eqref{eq: ratio control}.  For any predictable  non-negative process $(s,y)\mapsto H(s,y)$,
\begin{align*}
&\int_0^t\int_0^1 H(s,y) \mu(\rd s, \rd y)- \int_0^t\int_{|z|\le 1} H(s,|z(s)|^{\de\be_\nu}) c\nu(z)\rd z \rd s  \\
&=
\int_0^t\int_{|z|\le 1} \int_0^\infty \int_0^1 \indiq_{\bra{v(s)\le  \frac{\ph_0(M_{s-}+z(s))}{\ph_0(M_{s-})}, \, x(s) \le
\frac{c \ph_0(M_{s-})}{\ph_0(M_{s-}+z(s))}}} H(s, |z(s)|^{\de\be_\nu}) \wt N_{p'}(\rd s,\rd z, \rd v, \rd x)
\end{align*}
is a local martingale. Then the compensator of $\mu$ is $c\,\rd t\,\pi_\de(\rd y)$.
By \cite[Ch.2, Th.1.8]{JS03}, $\mu$ is a Poisson measure with intensity $c \,\rd t \, \pi_\de(\rd y)$.

\mk
\noindent
\emph{Step 2}: Applying the integral test of covering for limsup sets built from a Poisson measure,  see
\cite{B94,S72}, we only need to show that
  \begin{align*}
  \int_0^1 \exp\( 2\int_t^1 c\pi_\de((y,1)) \rd y\) \rd t = \infty.
  \end{align*}
The divergence of this integral can be proved by a modification of \cite[Lem. 2]{J99}. Note that
\begin{align*}
\int_t^1 c\pi_\de((y,1))\rd y = c\int_{t^{1\over \de\be_\nu}}^1\left(\int_{u<|x|<1}\nu(x)\rd x\right)
\de\be_\nu u^{\de\be_\nu -1} \rd u.
\end{align*}
Write $C_j= \int_{2^{-j-1}<|x|\le 2^{-j}}\nu(x)\rd x$ and $\omega(u)= \int_{u<|x|<1}\nu(x)\rd x$. Let $j(t)$ be the
unique integer such that $2^{j(t)-3}< t^{\frac{1}{\de\be_\nu}}\le 2^{-j(t)-2}$. Then we have
\begin{align*}
\int_{t^{1\over \de\be_\nu}}^1\omega(u)\de\be_\nu u^{\de\be_\nu -1} \rd u &\ge
\int_{2^{j(t)-2}}^{2^{j(t)-1}} \omega(u)\de\be_\nu u^{\de\be_\nu -1} \rd u \\
&\ge C_{j(t)} \de\be_\nu (2^{-j(t)-2})^{\de\be_\nu-1}2^{-j(t)-2} \\
&= C_{j(t)} \de\be_\nu (2^{-j(t)-2})^{\de\be_\nu}.
\end{align*}
By the definition of $\be_\nu$, for any $r\in (\de\be_\nu, \be_\nu)$, there exist infinitely many $j$ such that
$C_j\ge 2^{rj}$. For any such $j$ we have
\begin{align*}
\int_{2^{-(j+3)\de\be_\nu}}^{2^{-(j+2)\de\be_\nu}} \exp\left( 2\int_t^1 c\pi_\de((y,1)) \rd y \right) \rd t \ge
(2^{\de\be_\nu} -1) 2^{-(j+3)\de\be_\nu} \exp (c\de\be_\nu 2^{1-2\de\be_\nu} 2^{j(r-\de\be)}),
\end{align*}
which is bounded from below by $1$ for all $j$ sufficiently large.
\end{proof}

The latter lemma is a uniform approximation property of every time by the jumps. It is clear that $A_\de$ is monotone in $\de$,
while the critical value is $\de = 1$ for which the limsup set may or may not cover the semi-axis. As soon as $\de<1$, full
covering occurs. In particular, whenever $\de<1$, for every $t\ge 0$ there exist infinitely many $s_n\in J$ with $|z(s_n)|
\downarrow 0$, such that
\begin{align*}
|t- s_n| \le |z(s_n)|^{\be_\nu\de}.
\end{align*}

For fixed time-points, one might expect an improved inequality to hold, which motivates  the notion of the pointwise approximation
rate defined below.
\begin{definition}
Let $(t_n, r_n) \in \rr^+ \times \rr_*$ be a family of points.  We call
\begin{align}
\label{ineq1}
\de_t = \sup\left\{\de\ge 0:  |t-t_n|\le r_n^{\be_\nu\de} \, \mbox{\rm infinitely often} \right\}
\end{align}
the \emph{approximation rate} of $t\in \rr^+$ by the family of points.
\end{definition}
\noindent
The approximation rate is crucial in investigating the pointwise H\"older exponent of jump processes. By the covering lemma,
for all $t\ge 0$ we have $\de_t\ge 1$, almost surely. The use of this concept will appear clearly in the upper estimate of
$H_M(t)$ below.

\begin{proposition}
\label{upb}
For all $t\ge 0$,
\begin{align*}
H_M(t) \le \frac{1}{\be_\nu\de_t}
\end{align*}
almost surely.
\end{proposition}
\begin{proof}
Take any $t\in A_\de$. An application of Lemma \ref{jumplemma} to $\pro M$ and the set of $s_n$ in \eqref{ineq1} implies that $H_M(t)
\le 1/(\be_\nu\de)$. For an arbitrary $t$, we have the following cases. If $\de_t<\infty$, then for any $\e>0$, $t\in A_{\de_t-\e}$,
we have $H_M(t)\le 1/(\be_\nu(\de_{t}-\e))$. Letting $\e\to 0$ gives the result. If $\de_t= \infty$, then $t\in \cap_{\de\ge 1}
A_\de$, and  thus $H_M(t)=0$, which is the claimed upper bound.
\end{proof}

To derive a lower bound, we need to control the increments of the sample paths. An analogue of the following result appears in
\cite{B14} for L\'evy processes, however, since the GST processes have position-dependent increments, we need a substantial
upgrading. For each $n\in \nn$, write
\begin{align*}
Y_n(t) = \int_0^t \int_{ |z|\le 2^{-\frac{n}{\de\be_\nu}}}
\int_0^\infty \indiq_{\bra{v\le \frac{\ph_0(M_{s-}+z)}{\ph_0(M_{s-})}}} z \wt N(\rd s, \rd z, \rd v).
\end{align*}

\begin{lemma}
\label{estimate}
Let $\de>1$.   There exist finite constants $c_1, c_2 > 0$ such that for all $n\in \nn$,
\begin{align*}
\pp\left(\sup_{s,t\in [0,1], \ |s-t|\le 2^{-n}}|Y_n(t)-Y_n(s)| \ge 3 n 2^{-\frac{n}{\de\be_\nu}}  \sqrt{d}\right) \le c_1 e^{-c_2n}.
\end{align*}
\end{lemma}
\begin{proof}
Let $I_{n,k} = [k2^{-n}, (k+1)2^{-n})$. Using a dyadic approximation, the required probability can be bounded from above by
\begin{align}\label{eqsum}
\sum_{k=0}^{2^n-1} \pp\left(\sup_{t\in I_{n,k}} |Y_n(t) - Y_n(k 2^{-n})|\ge n   2^{-\frac{n}{\de\be_\nu}} \sqrt{d}\right).
\end{align}
We estimate the sum term by term. For each $k$, consider the semimartingale
$$
\wt Y_n(t) = 2^{\frac{n}{\de\be_\nu}}\left( Y_n(t+ k2^{-n})-Y_n(k2^{-n}\right), \quad t\in I_{n,k}.
$$
Applying It\^o's formula with the map $x\mapsto e^{x\cdot \xi}$, where $\xi =\vec{e_i}$ is the canonical orthonormal basis
of $\rr^d$, we obtain
\begin{eqnarray*}
e^{\wt Y_n(t)\cdot \xi}
&=&
1 +  \int_{k2^{-n}}^t \int_{ |z|\le 2^{-\frac{n}{\de\be_\nu}}}\int_0^{\infty} e^{\wt Y_n(s-)\cdot\xi}
\left(\exp\left( 2^{\frac{n}{\de\be_\nu}}\indiq_{\bra{v\le \frac{\ph_0(M_{s-}+z)}{\ph_0(M_{s-})}}} z \cdot \xi\right)-1\right)
\wt N(\rd s, \rd z, \rd v) \\
&& \qquad
+ \int_{k2^{-n}}^t \int_{ |z|\le 2^{-\frac{n}{\de\be_\nu}}}\int_0^{\infty} e^{\wt Y_n(s-)\cdot\xi}
\left(\exp\left( 2^{\frac{n}{\de\be_\nu}} \indiq_{\bra{v\le \frac{\ph_0(M_{s-}+z)}{\ph_0(M_{s-})}}} z \cdot \xi\right)-1\right.\\
&& \left.
\hspace{5.6cm}   - 2^{\frac{n}{\de\be_\nu}}\indiq_{\bra{v\le \frac{\ph_0(M_{s-}+z)}
{\ph_0(M_{s-})}}} z \cdot \xi \right) \rd v \nu(z)\rd z \rd s,
\end{eqnarray*}
for all $t\in I_{n,k}$.  Define the stopping times $\tau_r= \inf\{ t\in I_{n,k}: |\wt Y_n(t)|\ge r\}$, $r\in\nn$, with the
convention that $\inf \emptyset = \infty$.  By the c\`adl\`ag property of sample paths, $\tau_r\to \infty$ as $r\to\infty$, almost
surely.  Since the stopped compensated Poisson integral is a centered martingale, on taking expectation in the above formula and
using $|e^u-1-u|\le u^2$ for $|u|\le 1$, we get
\begin{align*}
\E[e^{\wt Y_n(t\wedge \tau_r)\cdot\xi}] &\le 1  +  \E\Big[ \int_{k2^{-n}}^{t\wedge\tau_r} \int_{|z|
\le
2^{-\frac{n}{\de\be_\nu}}}\int_0^{\infty} e^{\wt Y_n(s-)\cdot\xi} \indiq_{\bra{v\le \frac{\ph_0(M_{s-}+z)}{\ph_0(M_{s-})}}}
\( 2^{\frac{n}{\de\be_\nu}}z \cdot \xi \) ^ 2 \rd v \nu(z)\rd z \rd s\Big]   \\
&=
1 + \E \Big[  \int_{k2^{-n}}^{t\wedge\tau_r} \int_{|z|\le 2^{-\frac{n}{\de\be_\nu}}} \frac{\ph_0(M_{s-}+z)}{\ph_0(M_{s-})}
e^{\wt Y_n(s-)\cdot\xi} \( 2^{\frac{n}{\de\be_\nu}}z \cdot \xi \) ^ 2  \nu(z)\rd z \rd s\Big].
\end{align*}
Using the upper bound in \eqref{eq: ratio control}, we furthermore obtain
\begin{align*}
\E[e^{\wt Y_n(t\wedge \tau_r)\cdot\xi}]
&\le
1 + \frac{1}{c} \E \Big[  \int_{k2^{-n}}^{t\wedge\tau_r} e^{\wt Y_n(s-)\cdot\xi} \rd s  \int_{|z|\le 2^{-\frac{n}{\de\be_\nu}}}
2^{\frac{2n}{\de\be_\nu}} |z|^ {2-\de\be_\nu} |z|^{\de\be_\nu} \nu(z)\rd z \Big].
\end{align*}
The integral over $z$ in the expectation is bounded above by
\begin{align*}
 2^n \int_{|z|\le 1} |z|^{\de\be_\nu} \nu(z)\rd z = C_1 2^n,
\end{align*}
with a suitable  constant $C_1$, which does not depend on $n$ and is finite since $\de>1$. By Fubini's theorem,
\begin{align*}
\E[e^{\wt Y_n(t\wedge \tau_r)\cdot\xi}] \le 1+\frac{2^n}{c_1}\int_{k2^{-n}}^t \E[e^{\wt Y_n(s \wedge \tau_r)\cdot\xi}] \,\rd s,
\end{align*}
where $C_2 = C_1/c$. Gronwall's lemma yields then
\begin{align*}
\E[e^{\wt Y_n(t\wedge \tau_r)\cdot\xi}] &\le  e^{(t-k2^{-n})2^n/C_2} \le e^{1/C_2},
\end{align*}
for all $t\in I_{n,k}$. Letting $r \to \infty$ and using Fatou's lemma, we get $\E[e^{\wt Y_n(t)\cdot\xi}] \le  e^{1/C_2}$, and
similarly $\E[e^{- \wt Y_n(t)\cdot\xi}] \le  e^{1/C_2}$.  Hence, using that $ {x_1}^2+\cdots+{x_d}^2 \le d\max_i x_i^2$, we
have
\begin{align*}
\E[e^{|\wt Y_n(t)|/\sqrt{d}}] \le \sum_{i=1}^d \E [ e^{|\wt Y_n(t)\cdot \vec{e_i}|  } ] \le 2d e^{1/C_2}.
\end{align*}
To conclude, by the Markov inequality we see that each term in \eqref{eqsum} is bounded from above by $e^{-n} \E[e^{|\wt Y_n(t)|/\sqrt{d}}]$,
which is summable in $n$.
\end{proof}

\begin{remark}
\rm{
This lemma can be extended to any bounded interval. We thus focus on the unit interval $[0,1]$.
}
\end{remark}

We can now prove a lower bound for the H\"older exponent of the compensated Poisson integral
\begin{align*}
Y_t=   \int_0^t \int_{ |z|\le 1}\int_0^\infty \indiq_{\bra{v\le \frac{\ph_0(M_{s-}+z)}{\ph_0(M_{s-})}}}z \wt N(\rd s, \rd z, \rd v).
\end{align*}

\begin{proposition}
For all $t\in[0,1]\setminus J$,
\begin{align*}
H_Y(t) \ge \frac{1}{\de_t\be_\nu},
\end{align*}
almost surely.
\end{proposition}
\begin{proof}
The Borel-Cantelli lemma combined with Lemma \ref{estimate} give that for all $n$ larger than a suitable $n_0\in\nn$,
\begin{align*}
\sup_{s,t\in[0,1], \, |s-t|\le 2^{-n}}|Y_n(t)-Y_n(s)| \le 3\sqrt{d} n 2^{-\frac{n}{\de\be_\nu}}, \quad \mbox{a.s.}
\end{align*}
Fix a point of continuity $t\in[0,1]\setminus A_\de$. Let $s$ be close enough to $t$ so that for some $n<n_0$,
$$
2^{-n-1} < |t-s| \le 2^{-n}.
$$
Then
\begin{align*}
|Y_n(t)-Y_n(s)| \le 6\sqrt{d} \log\( \frac{1}{|s-t|}\) |t-s|^{\frac{1}{\de\be_\nu}}.
\end{align*}
Enlarging the value of $n_0$ if necessary, we see that $t\notin A_\de$ implies that any jump $s_p\in [s,t]$ satisfies
\begin{align*}
2^{-n}\ge |s-t| \ge |s_p-t|\ge |z(s_p)|^{\de\be_\nu},
\end{align*}
i.e., there are no jumps at the time-points $s_p\in J \cap [s,t]$ of size $|z(s_p)|\ge 2^{-\frac{n}{\de\be_\nu}}$. Hence,
\begin{eqnarray*}
\lefteqn{
\abs{ \int_s^t \int_{1\ge |z|\ge 2^{-\frac{n}{\de\be_\nu}}}\int_0^\infty
\indiq_{\bra{v\le \frac{\ph_0(M_{s-}+z)}{\ph_0(M_{s-})}}} z \wt N(\rd s, \rd z, \rd v) } } \\
&\qquad  =&
\abs{ \int_s^t \int_{1\ge |z|\ge 2^{-\frac{n}{\de\be_\nu}}}
\int_0^\infty \indiq_{\bra{v\le \frac{\ph_0(M_{s-}+z)}{\ph_0(M_{s-})}}} z \rd v\nu(z)\rd z \rd s } \\
& \qquad \le&
|s-t| \, \frac{\ph_0(M_{s-}+z)}{\ph_0(M_{s-})} \, \int_{1\ge |z|\ge 2^{-\frac{n}{\de\be_\nu}}} |z|\nu(z)\rd z.
\end{eqnarray*}
The integral over $z$ is bounded above by
\begin{align*}
 (2^{-\frac{n}{\de\be_\nu}})^{1-\de\be_\nu}\int_{|z|\le 1} |z|^{\de\be_\nu} \nu(z)\rd z \le  C |s-t|^{\frac{1}{\de\be_\nu} - 1},
\end{align*}
with a constant $C > 0$. Combining these estimates, we get
\begin{align*}
|Y_s-Y_t|
&\le |Y_n(t)-Y_n(s)| + \abs{ \int_s^t \int_{1\ge |z| \ge 2^{-\frac{n}{\de\be_\nu}}}
\int_0^\infty \indiq_{\bra{v\le \frac{\ph_0(M_{s-}+z)}{\ph_0(M_{s-})}}} z \wt N(\rd s, \rd z, \rd v) } \\
&\le c |t-s|^{\frac{1}{\de\be_\nu}} \log\frac{1}{|s-t|},
\end{align*}
where $c$ is a finite constant dependent on $M$ and $d$. Hence, almost surely, for all rational $\de>1$
we have $H_Y(t)\ge 1/(\de\be_\nu)$ at all times of continuity $t\in[0,1]\setminus A_\de$. By the definition of
$\de_t$, it is seen that $H_Y(t)\ge 1/(\de_t\be_\nu)$, for all continuity points $t\in[0,1]$, almost surely.
\end{proof}

\begin{remark}
\rm{
Using the argument in the proof of  Proposition \ref{upb}, we can similarly show $H_Y(t)\le 1/(\de_t\be_\nu)$ for
all $t$.
}
\end{remark}

\begin{theorem}
\label{th: holder exponent}
Under the assumptions of Theorem 1.3, for all times of continuity $t$,
\begin{align*}
 H_M(t) = \begin{cases} \frac{1}{\de_t\be_\nu}\wedge \frac{1}{2} & \mbox{ if } \sigma\neq 0 \\  \\
 \frac{1}{\de_t\be_\nu} & \mbox{ if } \sigma=0
 \end{cases}
\end{align*}
almost surely.
\end{theorem}
\begin{proof}
We distinguish three situations according to the matrix $\sigma$ and the value of $\be_\nu$.

\vspace{0.1cm}
\noindent
\emph{Case 1}: Let $\sigma\neq 0$ and $\be_\nu\in(0,2]$. Recall that for any $f,g: \RR^+\to \RR^d$ locally bounded
functions $H_{f+g}(t)\ge \min(H_f(t), H_g(t)$, where equality holds when the H\"older exponents of $f$ and $g$
are different at $t$.  Since the H\"older exponent of Brownian motion is $1/2$ everywhere and the drift terms
are differentiable at every $t$ (necessarily their H\"older exponent is larger or equal to $1$), we see that
the sum of Brownian motion and the two drifts has H\"older exponent equal to $1/2$ everywhere. The uncompensated
Poisson integral is locally constant, thus it does not influence the local regularity of $\pro M$ except on the
set of jump times (finite in any bounded interval). The compensated Poisson integral has H\"older exponent
$1/(\de_t\be_\nu)$ at any point of continuity $t$.  The claimed formula follows if $1/(\de_t\be_\nu)\neq 1/2$,
otherwise $1/2$ is a straightforward lower bound for $H_M(t)$, and it is also an upper bound due to Lemma
\ref{jumplemma}. Thus the identity follows.

\vspace{0.1cm}
\noindent
\emph{Case 2}: Let $\sigma=0$ and $\be_\nu\in [1,2]$. In this case, $1/(\de_t\be_\nu)\le 1$, since $\de_t\ge 1$ for all
$t$, i.e., the drifts are all smoother than the compensated Poisson integral. The result follows.

\vspace{0.1cm}
\noindent
\emph{Case 3}: Let $\sigma=0$ and $\be_\nu\in(0,1)$. Our assumption implies that the drift terms are smoother than the
compensated Poisson integral. To see this, note that for any locally bounded $f: \RR^+ \to \RR^d$, $g:\RR^d\to \RR^d$,
with $F(t)= \int_0^t g(f(s))\rd s$, we have that whenever $g\in C^k(\RR^d)$ with $k\ge H_f(t)$, it follows that
$H_F(t)\ge 1 + H_f(t)$. In particular, we have $H_F(t)> H_f(t)$. Applying this to $f=M$ and with $g$ chosen to be the
drift coefficient in \eqref{sde}, combined with the fact that $H_M(t)\le 1/(\de_t\be_\nu)\le 1/\be_\nu$, yields
$H_M(t)=H_Y(t)$, as claimed.

\vspace{0.1cm}
To complete the proof, in a concluding step we remove condition \eqref{eq: ratio control}.  Let
\begin{align*}
\Omega_{K,b} = \left\{\omega\in\Omega:  \sup_{t\le b} |M_t(\omega)| \le K \right\}.
\end{align*}
The c\`adl\`ag properties of the sample paths imply that $\pp(\Omega_{K,b})\to 1$ as $K \to \infty$.  Assumption
\ref{assump}
implies that the two-sided inequality in \eqref{eq: ratio control} holds uniformly for $|z|\le 1$, $|x|\le K$, for every
$K\in\nn_*$, with $c$ dependent on the value of $K$. For paths in $\Omega_{K,b}$, we have shown the result above. Letting
$K\to \infty$, then $b\to\infty$ completes the proof.
\end{proof}

\subsubsection{Proof of Theorem \ref{spectrum}: multifractal spectrum}
\label{sec: theo_spec}

We determine the  multifractal spectrum under condition \eqref{eq: ratio control}; the extension to the general situation
can be done as at the end of the last subsection.

Note that by Theorem \ref{th: holder exponent} it suffices to consider $h\in [0,1/\be_\nu]$. For every such $h$, we have that
\begin{align}
E_M(h)
&= \bra{t\ge 0:  \de_t= \frac{1}{h\be_\nu}} \setminus J =
\( \bigcap_{\al<1/(h\be_\nu)} {A_\al}\) \setminus  \( \bigcup_{\al>1/(h\be_\nu)} A_\al  \) \setminus J \nonumber\\
&= \( \bigcap_{n \geq 1} A_{1/(h\be_\nu)- 1/n}  \)  \setminus  \( \bigcup_{n \neq 1} A_{1/(h\be_\nu)+1/n}  \) \setminus J.
\label{eq: E_M(h)}
\end{align}

First we give an upper bound on the Hausdorff dimension of the family of sets $\{A_\de, \de\ge 1\}$. Observe that for any $j_0$,
\begin{align*}
A_\de \subset \bigcup_{j\ge j_0} \bigcup_{s\in J \atop 2^{-j-1}\le |z(s)|< 2^{-j}} (s-|z(s)|^{\be_\nu\de}, s+|z(s)|^{\be_\nu\de}).
\end{align*}
We can use these intervals as a covering system of $A_\de$. It suffices to show that for every $s> 1/\de$, almost surely,
\begin{align} \label{eq: up bound for spec}
\sum_{j\ge j_0} (2^{-j\be_\nu\de })^s N([0,1]\times [0,1/c]\times \{z: 2^{-j-1}\le |z|< 2^{-j}\}) < \infty,
\end{align}
 where $c$ is the constant in \eqref{eq: ratio control}. This implies that for all $\de\ge 1$ we have $\dimh A_\de\le 1/\de$,
 almost surely.

Next we  prove \eqref{eq: up bound for spec}.  Note that $N_j:=N([0,1]\times [0,1/c]\times \{z: 2^{-j-1}\le |z|< 2^{-j}\})$
is a Poisson random variable with parameter $C_j/c$, where $C_j = \int_{2^{-j-1}<|z|\le 2^{-j}}\nu(z)\rd z$ as in the previous
section.  Let $r\in (\be_\nu, \be_\nu\de s)$. Then by the  definition of $\be_\nu$ we have $C_j\le 2^{jr}$, for all $j$ large
enough. Hence by the Markov inequality,
 \begin{align*}
  \PP(N_j\ge 2\cdot 2^{jr}) \le \PP(|N_j-C_j|\ge 2^{jr}) \le 2^{-jr}.
 \end{align*}
It then follows by the Borel-Cantelli lemma that $N_j\le 2^{jr}$ almost surely, for all $j$ sufficiently large.  The convergence
of the series follows, since we choose $r < \be_\nu\de s$.

The above combined with \eqref{eq: E_M(h)} implies that $\dimh E_M(h)\le \frac{1}{\be_\nu\de}$. To obtain a lower bound on
the spectrum, we make use of the following general result; for a proof see \cite{J00}. Let $|\cdot|$ denote Lebesgue measure
in $\RR$.

\begin{theorem}
Let $(\la_n, \varepsilon_n)$ be a family of points, with $\la_n\in [0,1]$ and $\varepsilon_n>0$. Define $G_\de =
\limsup_{n\to\infty} (\la_n-\varepsilon_n^\de, \la_n +\varepsilon^\de)$.  If $|G_1|=1$, then for all $\de\ge 1$
$$
\mathcal{H}^{\phi_\de}(G_\de)>0,
$$
where $\phi_\de(x)= x^{1/\de} |\log x|^2$, and $\mathcal{H}^{\phi_\de}(E)$ is the Hausdorff measure of the set $E$ with respect
to the gauge function $\phi_\de$.
\end{theorem}

In order to apply the above result, one would need to prove the almost-covering $|A_1\cap [0,1]|=1$ at critical index $\de=1$,
which holds for regular L\'evy measures such as the isotropic $\al$-stable case with $\nu(z)= |z|^{-\al-d}$.  However, for
some ill-behaved L\'evy measures the situation $|A_1\cap [0,1]|<1$ may occur.  To overcome this, we use a trick from
\cite[Prop. 3.2]{X15} to construct a family of limsup sets $A^*_\de$ embedded in $A_\de$, in the sense that for every $\de'< \de$,
\begin{align}\label{eq: inclusion}
A_\de\subset A^*_\de\subset A_{\de'}
\end{align}
satisfying
\begin{align}\label{eq: A*1cover}
|A^*_1\cap[0,1]|=1.
\end{align}
Using \eqref{eq: inclusion}, we can express $E_M(h)$ in \eqref{eq: E_M(h)} with all $A_\de$ replaced by $A^*_\de$. A use of
\eqref{eq: A*1cover} then implies that $\cH^{\phi_\de} (A^*_\de)>0$ almost surely, for all $\de \ge 1$. Recalling that $\dimh
A^*_\de\le 1/\de$, we have $\cH^{\phi_{1/(h\be_\nu)}} (A^*_{1/(h\be_\nu)+1/n})=0$, which gives that $\cH^{\phi_{1/(h\be_\nu)}}
(E_M(h))>0$. This proves that  $D_M(h) \ge \be_\nu h$ almost surely, {simultaneously} for all $0\le h\le 1/\be_\nu$, as required.

To complete the argument, it remains to construct the sets $A^*_\de$ satisfying \eqref{eq: inclusion}-\eqref{eq: A*1cover}.
For all integers $m<n\le \infty$, let
\begin{align*}
A^{m,n}_\de = \bigcup_{s\in J \atop 2^{-n}\le |z(s)|< 2^{-m}} (s- |z(s)|^{\be_\nu\de}, s+ |z(s)|^{\be_\nu\de}).
\end{align*}
Set $m_1=1$. Due to Lemma \ref{lem: covering subcritical}, there exists $m_2>m_1$ such that 
$$
[0,1]\subset A_{1-\frac{1}{2}}\subset A^{m_1,\infty}_{1-\frac{1}{2}} \quad \mbox{and} \quad |A^{m_1,m_2}_{1-\frac{1}{2}}|\ge 
\frac{1}{2}.
$$ 
Similarly, there exists $m_3>m_2$ such that 
$$[0,1]\subset A_{1-\frac{1}{3}}\subset A^{m_2,\infty}_{1-\frac{1}{3}} \quad \mbox{and} \quad |A^{m_2,m_3}_{1-\frac{1}{3}}|\ge 
1-\frac{1}{3}.
$$
We define a sequence $(m_j)_{j\ge 1}$ inductively such that for all $j\ge 2$, $|A^{m_{j-1},m_j}_{1-1/j}|\ge 1-1/j$. Hence,
\begin{align*}
 \big|\limsup_{j\to\infty} A_{1-1/j}^{m_{j-1}, m_j} \cap [0,1] \big| \ge \limsup_{j\to\infty} \big|A_{1-1/j}^{m_{j-1}, m_j}
 \cap [0,1]\big| = 1.
\end{align*}
Define $A^*_\de = \limsup_{j\to\infty} A^{m_j,m_{j+1}}_{\de(1-1/j)}$; then the above formula shows \eqref{eq: A*1cover}. To
show property \eqref{eq: inclusion}, note that $A_\de = \limsup_{j\to\infty} A^{m_j,m_{j+1}}_{\de}$, and the first inclusion
then follows. The second inclusion also holds since for every $\de'<\de$ we have $\de'<\de(1-1/j)$, for all sufficiently large
$j$. This completes the proof.

\subsection{Concluding remarks}
As seen from the proof, the only case when we require some extra smoothness condition for the ground state is when the
Blumenthal-Getoor index is $\be_\nu<1$ and $\sigma=0$. As said in Remark \ref{gsreg} this is
known to hold in some cases, and it can be expected further to hold more widely.

We conjecture that the multifractal nature of a GST process will change if the ground state is less regular and
$\nabla\ln\ph_0$ is $C^\e$, with $1+\ep<1/\be_\nu$.  To see this, consider a simple representation for a GST
process when $\be_\nu<1$ and $\sigma=0$.  In such cases the process has finite variation, thus the compensated Poisson integral
can be decomposed in a difference of an uncompensated Poisson integral and a drift term. More precisely, the GST
process is a weak solution of the simple SDE with jumps
\begin{align*}
M_t = M_0+ \int_0^t b(M_s) \rd s + \int_0^t \int_{\RR^d_*} z N(\rd s, \rd z),
\end{align*}
where $N$ is a Poisson measure with intensity $\rd t\nu(z)\rd z$, and $b(x)= \nabla\ln\ph_0(x)+
(\int_{|z|\le 1} z\nu(z)\rd z)x$ is the drift coefficient. Recall that the H\"older exponent of the pure jump L\'evy
term is equal to $1/(\de_t\be_\nu)$.

Define the point processes
\begin{align*}
 p &= \{ (s, z(s)); s\in D \} \\
\wt p &= \{ (s, r(s)); s\in D \} \;\;\mbox{with}\;\; r(s) = b(M_{s-}+z(s))-b(M_{s-}),
\end{align*}
and the associated approximation rates $\de_t$ (for $p$) and $\wt\de_t$ (for $\wt p$) as in \eqref{ineq1}. The H\"older
exponent of the drift term will depend on $\wt p$. Indeed, for the process $b(M_t)=G_t$ Lemma \ref{jumplemma} implies
$H_G(t) \le 1/(\be_\nu\wt\de_t)$.  When $b$ is only $C^\e$, the jump size $r(s)\le z(s)^\e$. Therefore, $H_G(t)\le
\e/(\be_\nu\de_t)$ whenever $r(s)\asymp z(s)^\e$ occurs for infinitely many $s$ tending to $t$.

On the other hand, in \cite{B14} it is shown that when the H\"older exponent of a L\'evy process is less than
$1/(2\be_\nu)$ at some point $t$, or equivalently $\de_t> 2$, the time $t$ can not be an oscillating singularity in
the sense that its primitive must have H\"older exponent $1+ 1/(\de_t\be_\nu)$ at time $t$. It is tempting to expect
that the drift term here has H\"older exponent at most $1+ \e/(\de_t\be_\nu)$ as long as $\de_t>2$, for instance,
equal to $3$, and $r(s)\asymp z(s)^\e$ occurs for infinitely many $s$ tending to $t$. If such a $t$ exists, we get
$1+ \e/(3\be_\nu)< 1/(3\be_\nu)$ for $\e$  sufficiently small. This implies $H_M(t)\le 1+ \e/(3\be_\nu)$, and changes
the singularity sets $E_M(h)$ for $h\le 1+ \e/(3\be_\nu)$.

\bigskip
\noindent
\textbf{Acknowledgments: } JL thanks IHES, Bures-sur-Yvette, for a visiting fellowship, where part of this
paper has been written.  XY wishes to thank Prof. Yimin Xiao for stimulating discussions.

\end{document}